\pdfoutput=1
\documentclass[10pt]{amsart}
\usepackage[T1]{fontenc}
\usepackage[utf8]{inputenc}

\usepackage{amssymb}
\usepackage[colorlinks,pagebackref=true]{hyperref}
\renewcommand*\backref[1]{\ifx#1\relax \else (Cited on p.~#1) \fi}
\usepackage[inline]{enumitem}
\usepackage{cleveref}
\usepackage{caption}
\usepackage{subcaption}
\usepackage{tikz}

\theoremstyle{definition}
\newtheorem{theorem}{Theorem}[section]
\newtheorem{proposition}[theorem]{Proposition}
\newtheorem{corollary}[theorem]{Corollary}
\newtheorem{lemma}[theorem]{Lemma}
\newtheorem{definition}[theorem]{Definition}
\newtheorem{remark}[theorem]{Remark}
\newtheorem{question}[theorem]{Question}
\newtheorem{example}[theorem]{Example}
\AtBeginEnvironment{example}{%
  \pushQED{\qed}%
}
\AtEndEnvironment{example}{\popQED\endexample}

\numberwithin{equation}{section}

\title[Jets of monomial ideals and very well-covered graphs]{Jets and principal components of monomial ideals, and very well-covered graphs}
\author[F. Galetto]{Federico Galetto}
\address{Department of Mathematics and Statistics, Cleveland State University, 2121 Euclid Avenue, RT 1515, Cleveland, OH 44115-2215, USA}
\email{f.galetto@csuohio.edu}
\urladdr{https://math.galetto.org}
\thanks{F. Galetto was partially supported by the grant NSF DMS--2200844.}
\author[N. Iammarino]{Nicholas Iammarino}
\address{Department of Mathematics and Statistics, Cleveland State University, 2121 Euclid Avenue, RT 1515, Cleveland, OH 44115-2215, USA}
\email{nickiammarino@gmail.com}
\author[T. Yu]{Teresa Yu}
\address{Department of Mathematics, University of Michigan, 530 Church Street, Ann Arbor, MI 48109-1043, USA}
\email{twyu@umich.edu}
\urladdr{https://sites.google.com/view/teresayu}
\thanks{T. Yu was supported by the grant NSF DGE--2241144.}
\date{June 2024}
\keywords{jet, principal component, monomial ideal, graph, clutter, very well-covered, Hilbert series, Betti numbers}
\subjclass[2020]{Primary 05E40; Secondary 14E18, 05C65, 05C69, 13F55}

\begin{document}

\begin{abstract}
  Motivated by using combinatorics to study jets of monomial ideals,
  we extend a definition of jets from graphs to clutters. We offer
  some structural results on their vertex covers, and show an
  interesting connection between the cover ideal of the jets of a
  clutter and the symbolic powers of the cover ideal of the original
  clutter. We use this connection to prove that jets of very
  well-covered graphs are very well-covered. Next, we turn our
  attention to principal jets of monomial ideals, describing their
  primary decomposition and minimal generating sets. Finally, we give
  formulas to compute various algebraic invariants of principal jets
  of monomial ideals, including their Hilbert series, Betti numbers,
  multiplicity and regularity.
\end{abstract}

\maketitle

\section{Introduction}
\label{sec:introduction}

The tangent bundle of a scheme $X$ can be viewed as a collection of
pairs $(p,v)$ where $p$ is a point of $X$ and $v$ is a tangent vector
to $X$ at $p$.  A vector $v$ is tangent to $X$ at $p$ if it satisfies
some algebraic condition that can be expressed in terms of first-order
differentials. As one knows from calculus, it is possible to improve
upon a linear approximation by using higher order derivatives, and
this is roughly the idea behind jet schemes. An \emph{$s$-jet} of $X$ can be
thought of as a collection $(p,v_1,\dots,v_s)$ where $p$ is a point of
$X$, $v_1$ is a tangent vector to $X$ at $p$, and $v_s$ is a
``vector'' providing an order $s$ differential approximation of $X$
compatible with lower order approximations $p,v_1,\dots,v_{s-1}$. The
collection of all $s$-jets, denoted $\mathcal{J}_s (X)$, has a natural
scheme structure and can therefore be studied with geometric
tools. The concept of jets is attributed to C.~Ehresmann
\cite{MR0063123} (see also \cite[Chapter 4]{MR1337276}) and was first
introduced in the context of differential geometry. J.~Nash was among
the proponents of jets as a tool to study singularities in algebraic
geometry \cite{MR1381967}. Since then, jets have found applications to
a range of topics from motivic integration to birational geometry
\cite{MR1905328,MR1856396,MR1896234,MR2483946}.

If $X$ is a nonsingular variety, then the jet schemes
$\mathcal{J}_s (X)$ are nonsingular varieties (see \cite[Corollary
2.11]{MR2483946}). However, if $X$ has singularities, then its jet
schemes are often reducible (see \cite[Example 2.13]{MR2483946}). A
general description for the components of jet schemes is unknown, but
significant progress has been achieved for specific families of
varieties. Thanks to the work of several authors
\cite{MR2100311,MR2166800,MR2389248,MR3020097,MR3990994}, we have some
understanding of jet schemes for many determinantal varieties of
generic and symmetric matrices. This includes knowing when the jet
schemes are reducible, the number of components and their dimension,
and possibly additional information such as the equations or the
isomorphism class of these components as algebraic varieties. Another
case where partial information is available is that of monomial
schemes. R.A.~Goward and K.E.~Smith \cite{MR2229478} show that the
reduced scheme structure on the jets of monomial schemes is a monomial
scheme, and they explicitly describe monomial generators for the
ideals corresponding to these schemes. In addition, they identify the
components of the jets of a monomial hypersurface.

Motivated by the work of Goward and Smith, F.~Galetto, E.~Helmick, and
M.~Walsh \cite{MR4384023} defined the notion of \emph{jet graphs} in order to
use combinatorics to further study jets of monomial ideals. The idea
is the following: consider a graph $G$, and construct its edge ideal
$I(G)$; the radical of the order $s$ jets of $I(G)$ is the edge ideal
of a graph $\mathcal{J}_s (G)$, the graph of $s$-jets of $G$. The
minimal vertex covers of $\mathcal{J}_s (G)$ correspond to the minimal
primes of its edge ideal, which are also the minimal primes of the
jets of the ideal $I(G)$. Thus, one can use combinatorial data (the
vertex covers of a graph) to study components of the jet scheme of a
monomial ideal. The three authors offer preliminary results on vertex
covers of jet graphs \cite[Propositions 5.2 and 5.3]{MR4384023},
including a complete description for the vertex covers of the jets of
a complete bipartite graph $K_{n,n}$ \cite[Theorem
5.5]{MR4384023}. This last result provides evidence for their
conjecture that the jets of very well-covered graphs are very
well-covered \cite[Conjecture 6.5]{MR4384023}.

The first part of this manuscript builds upon and expands on the
earlier work of Galetto, Helmick, and Walsh by:
\begin{itemize}
\item extending the definition of jets to clutters (i.e., simple
  hypergraphs) (\S \ref{sec:background});
\item establishing structural results on vertex covers of jets of
  clutters (\S \ref{sec:vertex-covers-jets});
\item showing an interesting connection between the cover ideals of
  the jets of a clutter and the symbolic powers of the cover ideal of
  the clutter (\S \ref{sec:cover-ideals-jets});
\item and proving the conjecture on jets of very well-covered graphs
  (\S \ref{sec:very-well-covered}).
\end{itemize}
After we presented a preliminary version of this work, Tài Huy Hà and
Keri Sather-Wagstaff informed us of potential overlap with work of
S.A.~Seyed Fakhari \cite{MR3723124}. Motivated by the study of ideals
with linear resolutions whose (symbolic) powers also have linear
resolutions, the author starts from a graph $G$ and constructs,
without any apparent awareness of the relation to jets, a family of
graphs $G_k$ which are in fact the jet graphs $\mathcal{J}_{k-1}
(G)$. Where Seyed Fakhari shows that polarizing the $k$-th symbolic
power of the cover ideal of $G$ gives the cover ideal of $G_k$
\cite[Lemma 3.4]{MR3723124}, we independently prove \Cref{thm:3},
which achieves the same result for all clutters and establishes a
bijection between the minimal generators of the two ideals. Moreover,
Seyed Fakhari directly shows that when $G$ is very well-covered so are
the graphs $G_k$ \cite[Proposition 3.1(a)]{MR3723124}, thus also
verifying the conjecture of Galetto, Helmick, and Walsh. Meanwhile,
our \Cref{thm:6} offers an independent proof of the conjecture based
on a procedure that allows one to recover the vertex covers of the
jets of an arbitrary graph (see \Cref{exa:3}).  We believe that our
article contains significant generalizations on Seyed Fakhari's work,
and that it provides an interesting new perspective by uncovering a
previously unknown connection between jets and symbolic powers.

Another object of interest concerning jet schemes is the so-called
\emph{principal component}. For a variety $X$, the principal component
$\mathcal{P}_s (X)$ is the closure of the set of jets supported over
the smooth locus of $X$. When $X$ is irreducible, the principal
component $\mathcal{P}_s (X)$ is an irreducible component of the jet
scheme $\mathcal{J}_s (X)$ of the same dimension as
$\mathcal{J}_s (X)$ (see the discussion following \cite[Corollary
2.11]{MR2483946}). Moreover, the natural projection
$\mathcal{J}_s (X) \to X$ restricts to a dominant map
$\mathcal{P}_s (X) \to X$. Thus, $\mathcal{P}_s (X)$ is the
``principal'' component of $\mathcal{J}_s (X)$ in the sense that it
covers a large and significant portion of the jet scheme.  When $X$ is
the determinantal variety defined by the $2\times 2$ minors of a
generic matrix, T.~Košir and B.A.~Sethuraman proved that
$\mathcal{J}_1 (X)$ has only two components: the principal component
$\mathcal{P}_1 (X)$ and a component isomorphic to an affine space
\cite[Proposition 3.3]{MR2166800}. In this case, S.~Ghorpade,
B.~Jonov, and B.A.~Sethuraman showed that Hilbert series of
$\mathcal{P}_1 (X)$ is the square of the Hilbert series of $X$
\cite[Theorem 18]{MR3270176}. Inspired by these results, the second
part of our manuscript contains:
\begin{itemize}
\item a description for the principal components of monomial ideals
  (\S \ref{sec:principal-jets});
\item and formulas to compute the Hilbert series and Betti numbers of
  the principal components of monomial ideals (\S
  \ref{sec:invar-princ-jets}).
\end{itemize}
More specifically, \Cref{thm:7} describes the primary decomposition of
the principal components of a squarefree monomial ideal in terms of
the vertex covers of the corresponding clutter. Then, \Cref{cor:2}
gives a minimal generating set for the ideals of those principal
components which is similar to Goward and Smith's result for jets of
monomial schemes. To compute Hilbert series of principal components of
monomial ideals, we rely upon the connection between the Hilbert
series of a squarefree monomial ideal and the $f$-vector of the
simplicial complex associated to that ideal via Stanley--Reisner
theory. \Cref{thm:8} shows that $f$-vectors coming from principal
components of monomial ideals depend linearly upon the $f$-vector
coming from the original ideals. This shows that, although the simple
squaring formula of Ghorpade, Jonov, and Sethuraman does not apply to
monomial ideals, the Hilbert series of principal components can still
be easily computed from the same invariant of the original ideal. We
show a similar linear dependence also holds for the Betti numbers of
principal components of monomial ideals in \Cref{thm:9}.

\subsection*{Acknowledgements}
We are grateful for the feedback on this work that we received from
many mathematicians, including Jack Huizenga, Tài Huy Hà, Keri
Sather-Wagstaff, Alexandra Seceleanu, Joseph Skelton, Ivan Soprunov,
Zach Teitler, and Adam Van Tuyl. Part of this work began during the
2021 Michigan Research Experience for Graduates (MREG), which was
partially supported by the Univeristy of Michigan Rackham Graduate
School through its Faculty Allies for Diversity program.

\section{Background on Clutters and Jets}
\label{sec:background}

In this section, we provide preliminaries on clutters and introduce their jets, which are the main combinatorial objects of study in this paper. We also provide background on edge ideals and ideals of covers for clutters. We refer the reader to \cite{MR3362802} for further background on these topics.

A \emph{clutter}
with vertex set $X$ is a family of subsets of $X$, called edges, none
of which is included in another. The set of vertices and edges of a
clutter $C$ are denoted $V(C)$ and $E(C)$, respectively. Clutters
are also referred to as simple hypergraphs in the literature.

Let $\Bbbk [X]$ denote the ring of polynomials with variables given by
the elements of $X$, and coefficients in the field $\Bbbk$. Given a
clutter $C$ with vertex set $X$, the \emph{edge ideal} of $C$, denoted
$I(C)$, is the ideal of $\Bbbk [X]$ generated by all monomials
$x_1 \cdots x_r$ such that $\{x_1,\dots,x_r\}\in E(C)$; note that the given generating
set is minimal. The assignment $C\mapsto I(C)$ establishes a
one-to-one correspondence between the family of clutters and the
family of squarefree monomial ideals. It is worth mentioning that
squarefree monomial ideals are also in bijection with simplicial
complexes, which constitutes the basis for Stanley--Reisner theory; we
will make use of this connection in \S \ref{sec:invar-princ-jets}.

The notion of jets originates in geometry. For an introduction in the
context of algebraic geometry, the reader may consult \cite[\S
2]{MR2483946}. We present a definition of jets for ideals in a
polynomial ring. Our definition is compatible with the construction of
jet schemes of an affine variety starting from its defining ideal.

Given a set $X$ and $s\in\mathbb{N}$, let
\begin{equation*}
  \mathcal{J}_s (X) := \bigcup_{x\in X} \{x_0,\dots,x_s\}
\end{equation*}
and define a ring homomorphism
$\phi_s \colon \Bbbk [X] \to \Bbbk [\mathcal{J}_s (X)] [t] / \langle
t^{s+1}\rangle$ by setting
\begin{equation*}
  \phi_s (x) := \sum_{j=0}^s x_j t^j
\end{equation*}
for every $x\in X$.

\begin{definition}\label{def:1}
  Let $I=\langle f_1,\dots,f_r\rangle$ be an ideal of $\Bbbk [X]$. The
  \emph{ideal of $s$-jets of $I$}, denoted $\mathcal{J}_s (I)$, is the
  ideal of $\Bbbk [\mathcal{J}_s (X)]$ generated by the polynomials
  $f_{i,j}$ determined by the equalities
  \begin{equation*}
    \phi_s (f_i) = \sum_{j=0}^s f_{i,j} t^j
  \end{equation*}
  for every $i\in \{1,\dots,r\}$.
\end{definition}

\begin{example}\label{exa:1}
  Consider the ideal $I=\langle xy,yz\rangle$ of $\Bbbk [x,y,z]$. We have
  \begin{equation*}
    \begin{split}
      \phi_2 (xy) &= (x_0+x_1 t +x_2 t^2)(y_0+y_1 t +y_2 t^2)\\
                  &= x_0 y_0 + (x_0 y_1 +x_1 y_0) t + (x_0 y_2 +x_1 y_1 + x_2 y_0) t^2,\\
      \phi_2 (yz) &= (y_0+y_1 t +y_2 t^2)(z_0+z_1 t +z_2 t^2)\\
                  &= y_0 z_0 + (y_0 z_1 +y_1 z_0) t + (y_0 z_2 +y_1 z_1 + y_2 z_0) t^2.
    \end{split}
  \end{equation*}
  Therefore, the ideals of first and second jets of $I$ are:
  \begin{equation*}
    \begin{split}
      \mathcal{J}_1 (I) = \langle &x_0 y_0,y_0 z_0, x_0 y_1 +x_1 y_0,
                          y_0 z_1 +y_1 z_0 \rangle.\\
      \mathcal{J}_2 (I) = \langle &x_0 y_0,y_0 z_0, x_0 y_1 +x_1 y_0,
                                    y_0 z_1 +y_1 z_0,\\
      &x_0 y_2 +x_1 y_1 + x_2 y_0,
                          y_0 z_2 +y_1 z_1 + y_2 z_0\rangle.
    \end{split}
  \end{equation*}
\end{example}

The following theorem paraphrases a result of R.A.~Goward and K.~Smith
\cite[Theorem 2.1]{MR2229478}.

\begin{theorem}\label{thm:1}
  If $I$ is a squarefree monomial ideal in $\Bbbk [X]$, then, for
  every $s\in \mathbb{N}$, $\sqrt{\mathcal{J}_s (I)}$ is a squarefree
  monomial ideal of $\Bbbk [\mathcal{J}_s (X)]$. Moreover, the ideal
  $\sqrt{\mathcal{J}_s (I)}$ is minimally generated by the monomials
  $x_{1,i_1} \cdots x_{r,i_r}$ such that $x_1 \cdots x_r$ is a minimal
  generator of $I$, and $i_1,\dots,i_r\in \mathbb{N}$ satisfy
  $i_1+\dots+i_r \leqslant s$.
\end{theorem}

\begin{remark}\label{rem:1}
  With the notation of \Cref{thm:1}, applying the construction of
  \Cref{def:1} to a minimal generator $x_1 \cdots x_r$ of $I$ gives
  \begin{equation*}
    \phi_s (x_1 \cdots x_r) = \sum_{j=0}^s
    \left( \sum_{i_1+\dots +i_r = j}  x_{1,i_1} \cdots x_{r,i_r} \right) t^j.
  \end{equation*}
  Therefore, the minimal generators of $\sqrt{\mathcal{J}_s (I)}$ are
  exactly the terms of the polynomials in the generating set of
  $\mathcal{J}_s (I)$ described in \Cref{def:1}.
\end{remark}

Given the correspondence between squarefree monomial ideals and
clutters, one can associate a clutter to this radical of the ideal of
$s$-jets.  Generalizing the definition of jet graphs given in \cite[\S
2]{MR4384023}, we define the jets of a clutter as follows.

\begin{definition}\label{def:2}
  Let $C$ be a clutter with vertex set $X$. For $s\in \mathbb{N}$, the
  \emph{clutter of $s$-jets of $C$}, denoted $\mathcal{J}_s (C)$, is
  the clutter with vertex set $\mathcal{J}_s (X)$ whose edge ideal is
  given by $\sqrt{\mathcal{J}_s (I(C))}$.
\end{definition}

\begin{example}\label{exa:2}
  Consider the graph $G$ with vertex set $\{x,y,z\}$ whose edge ideal
  is given by $I(G) = \langle xy,yz\rangle$ (the same ideal as in
  \Cref{exa:1}). By \Cref{thm:1}, we have
  \begin{equation*}
    \begin{split}
      \sqrt{\mathcal{J}_1 (I(G))} &= \langle x_0 y_0,y_0 z_0, x_0 y_1,x_1 y_0,
                          y_0 z_1,y_1 z_0 \rangle,\\
      \sqrt{\mathcal{J}_2 (I(G))} &= \langle x_0 y_0,y_0 z_0, x_0 y_1,x_1 y_0,
                          y_0 z_1,y_1 z_0, x_0 y_2,x_1 y_1,x_2 y_0,
                          y_0 z_2,y_1 z_1, y_2 z_0\rangle.\\
    \end{split}
  \end{equation*}
  The graphs $G, \mathcal{J}_1 (G)$, and $\mathcal{J}_2 (G)$ are
  pictured in \Cref{fig:1}.
  \begin{figure}[htb]
    \centering
    \begin{subfigure}[b]{0.3\textwidth}
      \centering
      \begin{tikzpicture}[xscale=1.2]
        \tikzset{vertex/.style = {shape=circle,draw,thick,inner sep=2pt}}
        \tikzset{edge/.style = {thick}}
        \tikzset{every label/.append style={font=\small}}

        \node[vertex] (x) at (1,0) [label=below:{$x$}] {};
        \node[vertex] (y) at (2,0) [label=below:{$y$}] {};
        \node[vertex] (z) at (3,0) [label=below:{$z$}] {};

        \draw[edge] (x) -- (y) -- (z);
      \end{tikzpicture}
      \caption{$G$}
    \end{subfigure}
    \hfill
    \begin{subfigure}[b]{0.3\textwidth}
      \centering
      \begin{tikzpicture}[xscale=1.2]
        \tikzset{vertex/.style = {shape=circle,draw,thick,inner sep=2pt}}
        \tikzset{edge/.style = {thick}}
        \tikzset{every label/.append style={font=\small}}

        \node[vertex] (x0) at (1,0) [label=below:{$x_0$}] {};
        \node[vertex] (y0) at (2,0) [label=below:{$y_0$}] {};
        \node[vertex] (z0) at (3,0) [label=below:{$z_0$}] {};
        \node[vertex] (x1) at (1,1) [label=above:{$x_1$}] {};
        \node[vertex] (y1) at (2,1) [label=above:{$y_1$}] {};
        \node[vertex] (z1) at (3,1) [label=above:{$z_1$}] {};

        \draw[edge] (x0) -- (y0) -- (z0);
        \draw[edge] (x1) -- (y0) -- (z1);
        \draw[edge] (x0) -- (y1) -- (z0);
      \end{tikzpicture}
      \caption{$\mathcal{J}_1 (G)$}
    \end{subfigure}
    \hfill
    \begin{subfigure}[b]{0.3\textwidth}
      \centering
      \begin{tikzpicture}[xscale=1.2]
        \tikzset{vertex/.style = {shape=circle,draw,thick,inner sep=2pt}}
        \tikzset{edge/.style = {thick}}
        \tikzset{every label/.append style={font=\small}}

        \node[vertex] (x0) at (1,0) [label=below:{$x_0$}] {};
        \node[vertex] (y0) at (2,0) [label=below:{$y_0$}] {};
        \node[vertex] (z0) at (3,0) [label=below:{$z_0$}] {};
        \node[vertex] (x1) at (1,1) [label=above:{$x_1$}] {};
        \node[vertex] (y1) at (2,1) [label=above:{$y_1$}] {};
        \node[vertex] (z1) at (3,1) [label=above:{$z_1$}] {};
        \node[vertex] (x2) at (1,2) [label=above:{$x_2$}] {};
        \node[vertex] (y2) at (2,2) [label=above:{$y_2$}] {};
        \node[vertex] (z2) at (3,2) [label=above:{$z_2$}] {};

        \draw[edge] (x0) -- (y0) -- (z0);
        \draw[edge] (x1) -- (y0) -- (z1);
        \draw[edge] (x0) -- (y1) -- (z0);
        \draw[edge] (x1) -- (y1) -- (z1);
        \draw[edge] (x2) -- (y0) -- (z2);
        \draw[edge] (x0) -- (y2) -- (z0);
      \end{tikzpicture}
      \caption{$\mathcal{J}_2 (G)$}
    \end{subfigure}
    \caption{A path graph, and its first and second jets.}
    \label{fig:1}
  \end{figure}
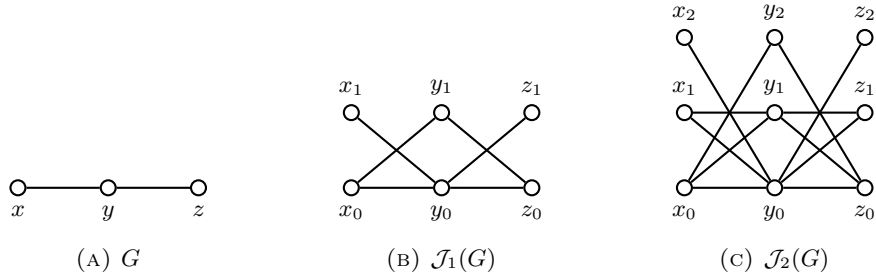
\end{example}

The following result describes the edges of a jet clutter, and is an
immediate consequence of Theorem \ref{thm:1}.

\begin{proposition}\label{pro:1}
  Let $C$ be a clutter with vertex set $X$, and let
  $s\in\mathbb{N}$. For every $\{x_1,\dots,x_r\} \subseteq X$ and
  $i_1,\dots,i_r\in \mathbb{N}$, the set
  $\{x_{1,i_1},\dots,x_{r,i_r}\}$ is an edge of the clutter
  $\mathcal{J}_s (C)$ if and only if $\{x_1,\dots,x_r\}$ is an edge of
  $C$ and $i_1+\dots+i_r \leqslant s$.
\end{proposition}

Let $C$ be a clutter with vertex set $X$.  A subset $W\subseteq X$ is
a \emph{vertex cover} of $C$ if every edge of $C$ contains at least
one vertex of $W$; if no proper subset of $W$ is a vertex cover of
$C$, then we say $W$ is \emph{minimal}.  We recall some ideas from combinatorial
commutative algebra that will be used throughout this work in relation
to vertex covers of a clutter.  The \emph{Alexander dual} of a clutter
$C$ with vertex set $X$, denoted $C^\vee$, is the clutter with the
same vertex set whose edges are the minimal vertex covers of $C$. The
edge ideal of $C^\vee$, denoted $I_c (C)$, is called the \emph{ideal
  of covers} of $C$. The
irredundant irreducible decomposition of $I_c (C)$ can be expressed in
terms of the edges of $C$ \cite[Theorem 6.3.39]{MR3362802}; namely, we
have
\begin{equation*}
  I_c (C) = \bigcap_{\{x_1,\dots,x_r\}\in E(C)} \langle x_1,\dots,x_r\rangle.
\end{equation*}
In addition to the ideal of covers $I_c (C)$, we will be interested in
its symbolic powers. The symbolic
powers of $I_c (C)$ have a concrete description that follows from
\cite[Proposition 4.3.25]{MR3362802} and the above decomposition of
$I_c (C)$; namely, for every $k\in \mathbb{N}$, we have
\begin{equation}\label{eq:1}
  I_c (C)^{(k)} = \bigcap_{\{x_1,\dots,x_r\}\in E(C)} \langle x_1,\dots,x_r\rangle^k.
\end{equation}
Similar to how monomials in $I_c (C)$ correspond to vertex covers of
$C$, monomials in $I_c (C)^{(k)}$ have an interesting combinatorial
interpretation. If $m \in I_c (C)^{(k)}$ is a monomial, then
\Cref{eq:1} implies that, for every edge $\{x_1,\dots,x_r\}$ of $C$,
there are natural numbers $i_1,\dots,i_r$ such that
$i_1+\cdots+ i_r \geqslant k$ and $x_1^{i_1}\cdots x_r^{i_r}$ divides
$m$. One can imagine a multiset $W$ containing each vertex $x\in X$ a
number of times equal to the exponent of $x$ in $m$; every edge of $C$
intersects this multiset in at least $k$ elements, i.e., the multiset
``covers'' each edge at least $k$ times. Rather than working with
multisets, we say a \emph{$k$-cover} of $C$ is the exponent vector of
a monomial $I_c (C)^{(k)}$ (see \cite[Definition 13.2.5]{MR3362802}
and preceding discussion for more details). In particular, a
\emph{minimal $k$-cover} of $C$ is the exponent vector of a minimal
generator of $I_c (C)^{(k)}$.

A $k$-cover of $C$ is called \emph{reducible} if there exist positive
integers $k_1,k_2$ with $k=k_1+k_2$, and monomials
$m_1\in I_c (C)^{(k_1)}, m_2\in I_c (C)^{(k_2)}$ such that
$m=m_1 m_2$. An \emph{irreducible} $k$-cover is one that is not
reducible. The irreducible
$k$-covers are the minimal generators of the symbolic Rees algebra of
$I_c (C)$
\begin{equation*}
  R_s (I_c (C)) = \bigoplus_{k\in \mathbb{N}} I_c (C)^{(k)} t^k
  \subset \Bbbk [X][t]
\end{equation*}
considered as a $\Bbbk$-algebra (see \cite[Corollary
13.2.19]{MR3362802}). This is a special case of the finite generation
of symbolic Rees algebras of monomial ideals due to J.~Herzog,
T.~Hibi, and N.V.~Trung \cite{MR2298826}.

\section{Vertex covers of jets}
\label{sec:vertex-covers-jets}

We use minimal vertex covers of a clutter $C$ to describe certain minimal vertex covers of the jets of the clutter $C$, as well as certain primary components of the ideal of $s$-jets $\mathcal{J}_s(I(C))$.

We begin with a basic structural property that all
minimal vertex covers of the jets of a clutter must satisfy.

\begin{lemma}\label{lem:1}
  Let $C$ be a clutter with vertex set $X$, and let $W$ be a minimal
  vertex cover of $\mathcal{J}_s (C)$ for some $s\in \mathbb{N}$. If
  $x\in X$ and $x_i \in W$ for some $i\in\{0,\dots,s\}$, then
  $x_j\in W$ for every $j\in \{0,\dots,i\}$.
\end{lemma}

\begin{proof}
  We proceed by induction on $s$. For $s=0$, the statement is
  trivial. Consider $s>0$ and assume the statement is true for all
  minimal vertex covers of $\mathcal{J}_{s-1} (C)$. Let $W$ be a
  minimal vertex cover of $\mathcal{J}_s (C)$. We will prove that if
  $x_i \in W$ for some $i\in\{1,\dots,s\}$, then $x_{i-1}\in W$; this
  is equivalent to the statement of the lemma.

  The set $W\cap \mathcal{J}_{s-1} (X)$ is a vertex cover of
  $\mathcal{J}_{s-1} (C)$ since
  $E(\mathcal{J}_{s-1} (C)) \subseteq E(\mathcal{J}_s (C))$. Let
  $W' \subseteq W\cap \mathcal{J}_{s-1} (X)$ be a minimal vertex cover
  of $\mathcal{J}_{s-1} (C)$.  If $x_i \in W'$, then
  $x_{i-1} \in W'\subseteq W$ by the inductive hypothesis. If
  $x_i \notin W'$, assume $x_{i-1} \notin W$.  Every edge of
  $\mathcal{J}_s (C)$ containing the vertex $x_i$ has the form
  $\{x_{i},y_{1,j_1},\dots,y_{r,j_r}\}$, where $\{x,y_1,\dots,y_r\}$
  is an edge of $C$ and $i+j_1+\dots+j_r \leqslant s$. Since
  $(i-1)+j_1+\dots+j_r \leqslant s-1$, we deduce that
  $\{x_{i-1},y_{1,j_1},\dots,y_{r,j_r}\}$ is an edge of
  $\mathcal{J}_{s-1} (C)$. Because we assumed that $x_{i-1}\notin W$
  and $W'\subseteq W$ is a vertex cover of $\mathcal{J}_{s-1} (C)$, we
  must have $y_{k,j_k}\in W$ for some index $k$. Hence, the edge
  $\{x_{i},y_{1,j_1},\dots,y_{r,j_r}\}$ is covered by the vertex
  $y_{k,j_k}$. Therefore, $W\setminus \{x_i\}$ is a vertex cover of
  $\mathcal{J}_s (C)$, which contradicts the minimality of $W$. We
  conclude that $x_{i-1} \in W$.
\end{proof}

\begin{remark}\label{rem:2}
  The work of Goward and Smith includes a complete description of the
  minimal vertex covers of the jets of a clutter with a single edge
  \cite[Theorem 2.2 (1)]{MR2229478}. We note that the condition
  imposed by \Cref{lem:1} is an implicit component of their
  description.
\end{remark}

Recall that given a set $X$ and $s\in\mathbb{N}$, we define
\begin{equation*}
  \mathcal{J}_s (X) := \bigcup_{x\in X} \{x_0,\dots,x_s\}.
\end{equation*}
The next result shows that applying this construction to a minimal
vertex cover of a clutter $C$ yields a minimal vertex cover of its jets, as well as describes certain primary components of the ideal of jets $\mathcal{J}_s(I(C))$.
This generalizes the graph-theoretic statements of \cite[Proposition 5.3]{MR4384023} and
  \cite[Theorem 2.7]{2201.04164} to clutters.

\begin{theorem}\label{thm:2}
  Let $C$ be a clutter with vertex set $X$, and let $W$ be a minimal
  vertex cover of $C$.
  \begin{enumerate}[label=(\arabic*)]
  \item\label{item:1} For every $s\in \mathbb{N}$, the set
    $\mathcal{J}_s (W)$ is a minimal vertex cover of
    $\mathcal{J}_s (C)$. Equivalently, for every $s\in \mathbb{N}$,
    the ideal $\langle \mathcal{J}_s (W) \rangle$ of
    $\Bbbk [\mathcal{J}_s (X)]$ is a minimal prime of
    $\mathcal{J}_s (I(C))$.
  \item\label{item:2} For every $s\in \mathbb{N}$, the ideal
    $\langle \mathcal{J}_s (W) \rangle$ of $\Bbbk [\mathcal{J}_s (X)]$
    is a primary component of $\mathcal{J}_s (I(C))$.
  \end{enumerate}
\end{theorem}


\begin{proof}
  \begin{enumerate*}
  \item Fix $s\in \mathbb{N}$. Consider an edge
    $\{x_{1,i_1},\dots,x_{r,i_r}\}$ of $\mathcal{J}_s (C)$. By
    \Cref{pro:1}, $\{x_1,\dots,x_r\}$ is an edge of $C$ and
    $i_1+\dots+i_r \leqslant s$. Since $W$ is a vertex cover of $C$,
    we have $x_j\in W$ for some $j$. It follows that
    $x_{j,i_j} \in \mathcal{J}_s (W)$. Hence, $\mathcal{J}_s (W)$ is a
    vertex cover of $\mathcal{J}_s (C)$.
  \end{enumerate*}

    Next, we show that $\mathcal{J}_s (W)$ is minimal. Suppose that a
    subset $W'\subseteq \mathcal{J}_s (W)$ is also a vertex cover of
    $\mathcal{J}_s (C)$. Since $W$ is a minimal vertex cover of $C$,
    for every $x\in W$, there is an edge $e$ of $C$ such that
    $e\cap W = \{x\}$. Writing $e = \{x,y_1,\dots,y_r\}$, we deduce
    that $\{x_s,y_{1,0},\dots,y_{r,0}\}$ is an edge of
    $\mathcal{J}_s (C)$ by \Cref{pro:1}. Since $W'$ is a vertex cover
    of $\mathcal{J}_s (C)$ contained in $\mathcal{J}_s (W)$ and none
    of the vertices $y_{j,0}$ belong to $\mathcal{J}_s (W)$, we deduce
    that $x_s \in W'$. By \Cref{lem:1}, we have
    $x_i \in \mathcal{J}_s (W)$ for all $i\in \{0,\dots,s\}$. In other
    words, $\mathcal{J}_s (W) \subseteq W'$. This implies that
    $W'=\mathcal{J}_s (W)$ and $\mathcal{J}_s (W)$ is minimal.

    Knowing that $\mathcal{J}_s (W)$ is a minimal vertex cover of
    $\mathcal{J}_s (C)$, it immediately follows that
    $\langle \mathcal{J}_s (W) \rangle$ is a minimal prime of
    $I(\mathcal{J}_s (C))$, the edge ideal of $\mathcal{J}_s (C)$
    \cite[Lemma 6.3.37]{MR3362802}. By \Cref{def:2}, we have
    $I(\mathcal{J}_s (C)) = \sqrt{\mathcal{J}_s (I(C))}$. Therefore,
    $\langle \mathcal{J}_s (W) \rangle$ is a minimal prime of
    $\mathcal{J}_s (I(C))$ by well-known results in commutative
    algebra (see, for example, \cite[Propositions 1.14 and
    4.6]{MR0242802}).

    \begin{enumerate*}[resume]
    \item For the remainder of this proof, let
    $R_s=\Bbbk [\mathcal{J}_s (X)]$, $J_s = \mathcal{J}_s (I(C))$ and
    $\mathfrak{p}_s = \langle \mathcal{J}_s (W)\rangle$. Let
    $\psi_s \colon R_s\to R_{s,\mathfrak{p}_s}$ be the canonical map
    to the localization of $R_s$ at the prime $\mathfrak{p}_s$. Since
    $\mathfrak{p}_s$ is a minimal prime of $J_s$, the ideal
    $\mathfrak{q}_s = \psi^{-1}_s (J_s R_{s,\mathfrak{p}_s})$ is the
    $\mathfrak{p}_s$-primary component of $J_s$ (see, for example,
    \cite[Theorem 6.8 (iii)]{MR1011461}). Clearly, we have
    $\mathfrak{q}_s\subseteq \sqrt{\mathfrak{q}_s} = \mathfrak{p}_s$.
    We need to show that $\mathfrak{p}_s\subseteq \mathfrak{q}_s$.
    \end{enumerate*}

    The ideal $\mathfrak{p}_s$ is generated by the variables
    $x_0,\dots,x_s$ with $x\in W$. Fixing $x$, we show that
    $x_0,\dots,x_s \in \mathfrak{q}_s$ by induction on $s$. When
    $s=0$, we have $x_0 \in \mathfrak{p}_0 = \mathfrak{q}_0$, since
    $J_0 \cong I(C)$ is a squarefree monomial ideal (hence radical)
    and $\mathfrak{p}_0 \cong \langle W\rangle$ is one of its minimal
    primes. Next, assume that for some $s\geqslant 0$, we have
    $x_0,\dots,x_s \in \mathfrak{q}_s$; we show that
    $x_0,\dots,x_s,x_{s+1} \in \mathfrak{q}_{s+1}$. The natural
    embedding of $R_s$ into $R_{s+1}$ sends $\mathfrak{q}_s$ into
    $\mathfrak{q}_{s+1}$, so $x_0,\dots,x_s \in
    \mathfrak{q}_{s+1}$. Since $W$ is a minimal vertex cover of $C$,
    there is an edge $e$ of $C$ such that $e\cap W = \{x\}$. Write
    $e = \{x,y_1,\dots,y_r\}$ and consider the minimal generator
    $x y_1 \cdots y_r$ of $I(C)$.  As observed in \Cref{rem:1}, the
    polynomial
    \begin{equation*}
      f = \sum_{i+j_1+\dots+j_r = s+1} x_i y_{1,j_1} \cdots y_{r,j_r} 
    \end{equation*}
    is a generator of $J_{s+1}$, and therefore
    $f\in \mathfrak{q}_{s+1}$. We can rewrite $f$ as
    \begin{equation*}
      f= x_{s+1} y_{1,0} \cdots y_{r,0} +
      \sum_{i=0}^s x_i \sum_{j_1+\dots+j_r = s+1-i} y_{1,j_1} \cdots y_{r,j_r},
    \end{equation*}
    so $x_{s+1} y_{1,0} \cdots y_{r,0} \in \mathfrak{q}_{s+1}$ because
    $x_0,\dots,x_s \in \mathfrak{q}_{s+1}$. Since
    $y_1,\dots,y_r \notin W$, we have
    $y_{1,0}, \dots, y_{r,0} \notin \mathfrak{p}_{s+1}$, hence the
    elements $y_{1,0}, \dots, y_{r,0}$ are invertible in
    $R_{s+1,\mathfrak{p}_{s+1}}$. We deduce that the fraction
    \begin{equation*}
      \frac{1}{y_{1,0} \cdots y_{r,0}}
      \psi_{s+1} (x_{s+1} y_{1,0} \cdots y_{r,0}) = \frac{x_{s+1}}{1}
    \end{equation*}
    belongs to $J_{s+1} R_{s+1,\mathfrak{p}_{s+1}}$, and
    $x_{s+1}\in \mathfrak{q}_{s+1}$. This concludes the proof of the
    inductive step. Therefore, for every $s\in \mathbb{N}$,
    $\mathfrak{q}_s = \mathfrak{p}_s$. Using the notation in the
    statement of the theorem, for every $s\in\mathbb{N}$, the primary
    component of $\mathcal{J}_s (I(C))$ corresponding to the minimal
    prime $\langle \mathcal{J}_s (W)\rangle$ is exactly
    $\langle \mathcal{J}_s (W)\rangle$.
\end{proof}

Let $C$ be a clutter with vertex set $X$. For any $s\in \mathbb{N}$,
the ideal $\mathcal{J}_s (I(C))$ defines a closed subscheme
$\mathcal{V}(\mathcal{J}_s (I(C)))$ of the projective space
$\mathbb{P}^{|\mathcal{J}_s (X)|-1}$. By \Cref{thm:2} \ref{item:1}, if
$W$ is a minimal vertex cover of $C$, then
$\langle\mathcal{J}_s (W) \rangle$ is a minimal prime of
$\mathcal{J}_s (I(C))$ and it defines a closed subscheme
$\mathcal{V}(\langle\mathcal{J}_s (W) \rangle)$ that is an irreducible
component of $\mathcal{V}(\mathcal{J}_s (I(C)))$. Recall that the
\emph{multiplicity} of $\mathcal{V}(\mathcal{J}_s (I(C)))$ along
$\mathcal{V}(\langle\mathcal{J}_s (W) \rangle)$ is the length of
$\Bbbk [\mathcal{J}_s (X)] / \mathcal{J}_s (I(C))$ localized at
$\langle\mathcal{J}_s (W) \rangle$ (see, for example, \cite[\S
4.3]{MR1644323}). Then, \Cref{thm:2} \ref{item:2} immediately implies
the following result.

\begin{corollary}
  Let $C$ be a clutter, and let $W$ be a minimal vertex cover of
  $C$. For every $s\in\mathbb{N}$, the multiplicity of
  $\mathcal{V}(\mathcal{J}_s (I(C)))$ along
  $\mathcal{V}(\langle\mathcal{J}_s (W) \rangle)$ is 1.
\end{corollary}

\begin{remark}
  When the clutter $C$ contains a single edge, the scheme defined by the
  ideal $I(C)$ is called a \emph{monomial hypersurface}.  The components
  of the jets of a monomial hypersurface were described by Goward and
  Smith (see \Cref{rem:2}), and the multiplicities of the jet schemes of
  a monomial hypersurface along their components were computed by
  C.~Yuen \cite{MR2372306}.
\end{remark}

\section{Cover Ideals of Jets and Symbolic Powers}
\label{sec:cover-ideals-jets}

In this section, we describe a connection between the cover ideals of
the jets of a clutter and the symbolic powers of the cover ideal of
the clutter. To some extent, this connection allows one to translate
results about symbolic powers into results about jets and vice versa.

\begin{definition}\label{def:3}
  Let $C$ be a clutter with vertex set $X$. For $s\in \mathbb{N}$,
  define a homomorphism of $\Bbbk$-algebras
  $\delta_s \colon \Bbbk [\mathcal{J}_s (X)] \to \Bbbk [X]$ by setting
  $\delta_s (x_i) = x$ for every $x\in X$ and $i\in \{0,\dots,s\}$. We
  call $\delta_s$ a \emph{depolarization} map.
\end{definition}

\begin{remark}\label{rem:3}
  The function $\delta_s$ in the statement of Definition \ref{def:3}
  reverses the operation known as \emph{polarization} of monomials
  (see e.g. \cite{MR2184792}), hence the name depolarization. We find
  depolarization more convenient to work with because it defines a
  ring homomorphism. Consider a monomial $m\in \Bbbk [X]$, and assume
  that for every $x\in X$ there exists a natural number
  $i\leqslant s+1$ such that $m$ is divisible by $x^i$ but not by
  $x^{i+1}$. In this case, the polarization of $m$ is the unique
  monomial in $\delta_s^{-1} (\{m\})$ that, for every $x\in X$, is
  divisible by $x_j$ for all $j\in \{0,\dots,i-1\}$.
\end{remark}

\begin{theorem}\label{thm:3}
  Let $C$ be a clutter with vertex set $X$. For every
  $s\in \mathbb{N}$, we have
  $\delta_s (I_c (\mathcal{J}_s (C))) = I_c (C)^{(s+1)}$, where
  $\delta_s \colon \Bbbk [\mathcal{J}_s (X)] \to \Bbbk [X]$ is the
  depolarization map. Moreover, $\delta_s$ induces a
  bijection between the set of minimal generators of
  $I_c (\mathcal{J}_s (C)))$ and the set of minimal generators of
  $I_c (C)^{(s+1)}$.
\end{theorem}

\begin{proof}
  We start by showing that
  $\delta_s (I_c (\mathcal{J}_s (C))) \subseteq I_c (C)^{(s+1)}$. To
  prove this inclusion, it is enough to show that
  $\delta_s (m) \in I_c (C)^{(s+1)}$ for every minimal monomial
  generator $m$ of $I_c (\mathcal{J}_s (C))$, because
  $I_c (\mathcal{J}_s (C))$ is a monomial ideal.  Fix a minimal
  monomial generator $m$ of $I_c (\mathcal{J}_s (C))$, and let $W$ be
  the corresponding minimal vertex cover of $\mathcal{J}_s (C)$.  To
  show that $\delta_s (m) \in I_c (C)^{(s+1)}$, we show that
  $\delta_s (m) \in \langle e\rangle^{s+1}$ for every edge $e$ of
  $C$. Fix an edge $\{x_1,\dots,x_r\}$ of $C$, and define, for every
  $i\in \{1,\dots,r\}$, the quantities
  \begin{equation*}
    u_i =
    \begin{cases}
      \max \{ j \in \mathbb{N} \colon x_{i,j} \in W \},
      & \text{if }x_{i,j}\in W \text{ for some } j\in\mathbb{N},\\
      -1,& \text{otherwise}.
    \end{cases}    
  \end{equation*}
  By Lemma \ref{lem:1}, the monomial
  \begin{equation*}
    \prod_{i=1}^r \prod_{j=0}^{u_i} x_{i,j} \in \Bbbk [\mathcal{J}_s (X)]
  \end{equation*}
  divides $m$, and hence the monomial
  \begin{equation}\label{eq:2}
    \prod_{i=1}^r x_i^{u_i+1} \in \Bbbk [X],
  \end{equation}
  obtained by applying $\delta_s$, divides $\delta_s (m)$. Suppose
  $\delta_s (m) \notin \langle x_1,\dots,x_r\rangle^{s+1}$; then the
  monomial in \eqref{eq:2} does not belong to
  $\langle x_1,\dots,x_r\rangle^{s+1}$. It follows that
  \begin{equation*}
    \sum_{i=1}^r (u_i+1) \leqslant s,
  \end{equation*}
  and, therefore, $\{x_{1,u_1+1},\dots,x_{r,u_r+1}\}$ is an edge of
  $\mathcal{J}_s (C)$ that is not covered by $W$. Since this is a
  contradiction, we conclude that
  $\delta_s (m) \in \langle x_1,\dots,x_r\rangle^{s+1}$, as desired.

  Next, we show that for every minimal monomial generator $m'$ of
  $I_c (C)^{(s+1)}$ there exists a minimal monomial generator $m$ of
  $I_c (\mathcal{J}_s (C))$ such that $\delta_s (m) = m'$. Given $m'$,
  we let $m$ be the polarization of $m'$ (see Remark \ref{rem:3}) so
  that $\delta_s (m) = m'$ holds true. To show
  $m\in I_c (\mathcal{J}_s (C))$, we show that $m\in\langle e\rangle$
  for every edge $e$ of $\mathcal{J}_s (C)$. Fix an edge
  $\{x_{1,i_1},\dots,x_{r,i_r}\}$ of $\mathcal{J}_s (C)$, where
  $\{x_1,\dots,x_r\}$ is an edge of $C$ and
  $i_1+\dots+i_r\leqslant s$.  For every $i\in\{1,\dots,r\}$, define
  the quantities
  \begin{equation*}
    v_i = \max \{ j\in\mathbb{N} \colon x_i^j\text{ divides } m'\}.
  \end{equation*}
  Since $x_1^{v_1}\cdots x_r^{v_r}$ divides $m'$ and
  $m' \in \langle x_1,\dots,x_r\rangle^{s+1}$, we get
  $v_1+\dots+v_r\geqslant s+1$.  If $x_{j,i_j} \nmid m$ for some
  $j\in\{1,\dots,r\}$, then $x_j^{i_j+1} \nmid m'$ by Remark
  \ref{rem:3}; hence $i_j+1>v_j$ or, equivalently, $i_j\geqslant v_j$.
  Thus, if none of the variables $x_{1,i_1},\dots,x_{r,i_r}$ divide
  $m$, then we have the inequalities
  \begin{equation*}
    \sum_{j=1}^r i_j \geqslant \sum_{j=1}^r v_j \geqslant s+1,
  \end{equation*}
  which is a contradiction. This implies that
  $m\in \langle x_{1,i_1},\dots,x_{r,i_r} \rangle$, and therefore
  $m\in I_c (\mathcal{J}_s (C))$, which concludes the proof of the
  equality $\delta_s (I_c (\mathcal{J}_s (C))) = I_c (C)^{(s+1)}$.
  
  Now, if $m$ is not a minimal generator of $I_c (\mathcal{J}_s (C))$,
  then there is a monomial $\tilde{m} \in I_c (\mathcal{J}_s (C))$ not
  equal to $m$ that divides $m$. It follows that
  $\delta_s (\tilde{m}) \in I_c (C)^{(s+1)}$, $\delta_s (\tilde{m})$
  divides $\delta_s(m)=m'$, and $\delta_s (\tilde{m}) \neq m'$, contradicting the
  fact that $m'$ is a minimal generator of $I_c (C)^{(s+1)}$. Thus,
  $m$ is a minimal generator of $I_c (\mathcal{J}_s (C))$, and the
  restriction of $\delta_s$ between the set of minimal generators of
  $I_c (\mathcal{J}_s (C)))$ and the set of minimal generators of
  $I_c (C)^{(s+1)}$ is surjective. The injectivity follows from Lemma
  \ref{lem:1}.
\end{proof}

For the combinatorially-minded reader, we offer the following
restatement of the bijection in Theorem \ref{thm:3}.

\begin{corollary}\label{cor:1}
  Let $C$ be a clutter. For every $s\in \mathbb{N}$, the minimal
  vertex covers of $\mathcal{J}_s (C)$ are in bijection with the
  minimal $(s+1)$-covers of $C$ via (de)polarization.
\end{corollary}

\begin{example}
  Consider the graph $G$ of \Cref{exa:2}. The ideal of covers of $G$ is
  \begin{equation*}
    I_c (G) = \langle x,y\rangle \cap \langle y,z\rangle
    = \langle xz,y\rangle,
  \end{equation*}
  and its symbolic square is
  \begin{equation*}
    \begin{split}
      I_c (G)^{(2)} = \langle x,y\rangle^2 \cap \langle y,z\rangle^2
    = \langle x^2 z^2, xyz,y^2\rangle,
    \end{split}
  \end{equation*}
  where the given generators are minimal. The polarization of
  $I_c (G)^{(2)}$ is the ideal
  \begin{equation*}
    \langle x_0 x_1 z_0 z_1, x_0 y_0 z_0, y_0 y_1\rangle,
  \end{equation*}
  and its generators correspond to the minimal vertex covers of
  $\mathcal{J}_1 (G)$.
\end{example}

\begin{remark}
  If a generating set for the symbolic Rees algebra $R_s (I_c (C))$ is
  known, one can use it in combination with \Cref{thm:3} to find the
  vertex covers of the clutter of $s$-jets $\mathcal{J}_s (C)$ for any
  given $s\in\mathbb{N}$. We illustrate this idea for graphs in \S
  \ref{sec:very-well-covered}.

  The literature contains results that explicitly describe generating
  sets for symbolic Rees algebras of squarefree monomial ideals. For
  example, \cite[Proposition 4.6]{MR2298826} applies to the edge ideal
  of the skeleton of a simplex, while \cite[Theorem 3.4]{MR4124655}
  applies to the edge ideal of a unicyclic graph. These results are
  for the symbolic Rees algebra $R_s (I(C))$ of the edge ideal of a
  clutter $C$, rather than its ideal of covers $I_c (C)$. Recalling
  that $I_c (C) = I(C^\vee)$, one may then use known generating sets
  for $R_s (I(C))$ in combination with \Cref{thm:3} to find the vertex
  covers of the jets of the Alexander dual of $C$.
\end{remark}

\section{Jets of Very Well-Covered Graphs}
\label{sec:very-well-covered}

Let $G$ be a graph with vertex set $X$ and assume $G$ has no isolated
vertex. We say $G$ is \emph{well-covered} (or \emph{unmixed}) if all
its minimal vertex covers have the same cardinality \cite[Definition
7.2.8]{MR3362802}. Furthermore, if this cardinality equals
$\frac{1}{2} |X|$, then $G$ is called \emph{very well-covered}. There
has been interest from the point of view of graph and complexity
theory in recognition of (very) well-covered graphs \cite{MR677051,
  MR1217991, MR3106446}. Such graphs have also been well-studied in
the context of combinatorial commutative algebra
\cite{MR2772171,MR2793950,MR4418794}.  Our goal is to show that the
jets of very well-covered graphs are very well-covered, which was
conjectured in \cite[Conjecture 6.5]{MR4384023}.

We begin with a result characterizing irreducible $k$-covers for graphs; equivalently, it describes generators of the symbolic Rees algebra $R_s(I_c(G))$. Let $G$ be a graph with vertex set $X$.  Given a set of vertices
$Y\subseteq X$, the \emph{neighbor set} of $Y$, denoted $N(Y)$, is the
set of vertices of $G$ that are adjacent to at least one vertex in
$Y$. Also, the \emph{induced subgraph} $G[Y]$ is the graph with vertex
set $Y$ whose edge set consists of all edges of $G$ with both
endpoints in $Y$.  The following result is a restatement of \cite[Theorem
13.4.5]{MR3362802}.

\begin{theorem}\label{thm:4}
  Let $G$ be a simple graph with vertex set $X$. The irreducible
  $k$-covers of $G$ are the exponent vectors of the monomials
  $m\in I_c (G)^{(k)}$ with the following descriptions:
  \begin{enumerate}[label=(\arabic*)]
  \item $k=0$ and $m=x$ for some $x\in X$;
  \item\label{item:3} $k=1$ and $m=\prod_{w\in W} w$ for some minimal vertex cover
    $W$ of $G$;
  \end{enumerate}
  and, if $G$ is not bipartite,
  \begin{enumerate}[resume,label=(\arabic*)]
  \item\label{item:4} $k=2$ and $m=\prod_{x\in X} x$;
  \item\label{item:5} $k=2$ and
    \begin{equation*}
      m=\prod_{v\in N(U)} v^2 \prod_{z\in X\setminus (U\cup N(U))} z
    \end{equation*}
  \end{enumerate}
  for some nonempty independent set $U$ of $G$ such that
  \begin{itemize}
  \item $N(U)$ is not a vertex cover of $G$,
  \item $X\neq U\cup N(U)$, and
  \item the induced subgraph $G[X\setminus (U\cup N(U))]$ has no
    isolated vertex and is not bipartite.
  \end{itemize}
\end{theorem}

\begin{example}\label{exa:3}
  Let $G$ be the graph in \Cref{fig:2}. The vertex covers of $G$ are
  $\{u,w,x\}$, $\{v,w,x\}$, $\{u,w,y\}$, $\{v,w,y\}$, $\{v,x,y\}$.
  Since $G$ is not bipartite, it has the irreducible 2-cover of type
  \ref{item:4} corresponding to the monomial $uvwxy$. It also has the
  irreducible 2-cover of type \ref{item:5} with independent set
  $U=\{u\}$ and corresponding to the monomial $v^2wxy$. Polarizing
  these two monomials gives minimal vertex covers of
  $\mathcal{J}_1 (G)$. Then, for every $s\in\mathbb{N}$, the minimal
  vertex covers of $\mathcal{J}_s (G)$ are obtained by ``stacking''
  these two vertex covers of $\mathcal{J}_1 (G)$ and the minimal
  vertex covers of $G$.
  
  \Cref{fig:3} shows a minimal vertex cover of $\mathcal{J}_2 (G)$
  obtained by stacking the polarization of $v^2wxy$ (in black) and
  $uwy$ (in gray).
    \begin{figure}[htb]
    \centering
    \begin{tikzpicture}[xscale=1.75,yscale=0.75]
      \tikzset{vertex/.style = {shape=circle,draw,thick,inner sep=2pt}}
      \tikzset{edge/.style = {thick}}
      \tikzset{every label/.append style={font=\small}}

      \node[vertex] (v) at (0,0) [label=below:{$u$}] {};
      \node[vertex] (w) at (1,0) [label=below:{$v$}] {};
      \node[vertex] (x) at (2,0) [label=below:{$w$}] {};
      \node[vertex] (y) at (3,1) [label=above:{$x$}] {};
      \node[vertex] (z) at (4,0) [label=below:{$y$}] {};

      \draw[edge] (v) -- (w) -- (x) -- (y) -- (z) -- (x);
    \end{tikzpicture}
    \caption{A non-bipartite graph $G$.}
    \label{fig:2}
  \end{figure}
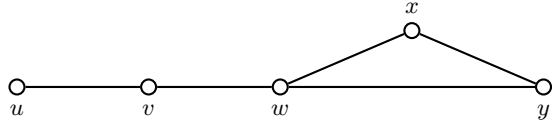
    \begin{figure}[htb]
    \centering
    \begin{tikzpicture}[xscale=1.75,yscale=0.75]
      \tikzset{vertex/.style = {shape=circle,draw,thick,inner sep=2pt}}
      \tikzset{edge/.style = {thick}}
      \tikzset{every label/.append style={font=\small}}

      \node[vertex,fill=lightgray] (v0) at (0,0) [label=below:{$u_0$}] {};
      \node[vertex,fill] (w0) at (1,0) [label=below:{$v_0$}] {};
      \node[vertex,fill] (x0) at (2,0) [label=below:{$w_0$}] {};
      \node[vertex,fill] (y0) at (3,1) [label=below:{$x_0$}] {};
      \node[vertex,fill] (z0) at (4,0) [label=below:{$y_0$}] {};
      \node[vertex] (v1) at (0,3) [label=above:{$u_1$}] {};
      \node[vertex,fill] (w1) at (1,3) [label=above:{$v_1$}] {};
      \node[vertex,fill=lightgray] (x1) at (2,3) [label=above:{$w_1$}] {};
      \node[vertex] (y1) at (3,4) [label=above:{$x_1$}] {};
      \node[vertex,fill=lightgray] (z1) at (4,3) [label=above:{$y_1$}] {};
      \node[vertex] (v2) at (0,6) [label=above:{$u_2$}] {};
      \node[vertex] (w2) at (1,6) [label=above:{$v_2$}] {};
      \node[vertex] (x2) at (2,6) [label=above:{$w_2$}] {};
      \node[vertex] (y2) at (3,7) [label=above:{$x_2$}] {};
      \node[vertex] (z2) at (4,6) [label=above:{$y_2$}] {};

      \draw[edge] (v0) -- (w0) -- (x0) -- (y0) -- (z0) -- (x0);
      \draw[edge] (v1) -- (w1) -- (x1) -- (y1) -- (z1) -- (x1);
      \draw[edge] (v0) -- (w1);
      \draw[edge] (v1) -- (w0);
      \draw[edge] (v0) -- (w2);
      \draw[edge] (v2) -- (w0);
      \draw[edge] (x0) -- (w1);
      \draw[edge] (x1) -- (w0);
      \draw[edge] (x0) -- (w2);
      \draw[edge] (x2) -- (w0);
      \draw[edge] (x0) -- (z1);
      \draw[edge] (x1) -- (z0);
      \draw[edge] (x0) -- (z2);
      \draw[edge] (x2) -- (z0);
      \draw[edge] (x0) -- (y1);
      \draw[edge] (x1) -- (y0);
      \draw[edge] (x0) -- (y2);
      \draw[edge] (x2) -- (y0);
      \draw[edge] (z0) -- (y1);
      \draw[edge] (z1) -- (y0);
      \draw[edge] (z0) -- (y2);
      \draw[edge] (z2) -- (y0);
    \end{tikzpicture}
    \caption{A reducible vertex cover of $\mathcal{J}_2 (G)$.}
    \label{fig:3}
  \end{figure}
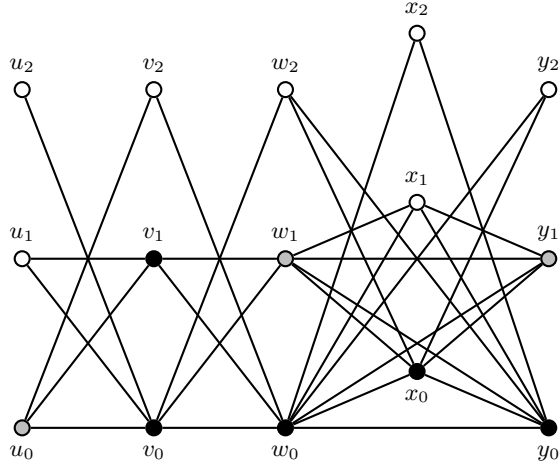
\end{example}

Next, we present some equivalent characterizations of very well-covered graph. Recall that a \emph{perfect matching} of
$G$ is a set of pairwise disjoint edges of $G$ whose union is $X$. Given a perfect matching $M$ and a vertex
$x$ of $G$, we denote $M(x)$ the unique vertex of $G$ adjacent to $x$
such that $\{x,M(x)\} \in M$. Favaron declares that $M$ satisfies the \emph{property (P)} if, for every $x\in X$, every neighbor of $x$ other
than $M(x)$ is not adjacent to $M(x)$, but is adjacent to all neighbors
of $M(x)$. The following result of O.~Favaron \cite[Theorem
1.2]{MR677051} characterizes very well-covered graphs in terms of the
property (P).
\begin{theorem}\label{thm:5}
  The following are equivalent.
  \begin{enumerate}
  \item $G$ is very well-covered.
  \item $G$ has a perfect matching that satisfies the property (P).
  \item $G$ has a perfect matching, and every perfect matching of $G$
    satisfies the property (P).
  \end{enumerate}
\end{theorem}

We finally come to the main result of this section.

\begin{theorem}\label{thm:6}
  Let $G$ be a very well-covered graph. Then, for every
  $s\in \mathbb{N}$, the graph of $s$-jets of $G$ is also very
  well-covered.
\end{theorem}

\begin{proof}
  Let $X$ be the vertex set of $G$. Since $G$ is very well-covered, $|X|=2n$ for some positive integer
  $n$.  For every $s\in \mathbb{N}$, the graph $\mathcal{J}_s (G)$ of
  $s$-jets of $G$ has vertex set $\mathcal{J}_s (X)$ with cardinality
  $(s+1)|X| = 2n(s+1)$. To show $\mathcal{J}_s (G)$ is very well-covered, we need to prove that every minimal vertex cover of
  $\mathcal{J}_s (G)$ has cardinality $n(s+1)$. Following from
  \Cref{thm:3} (and as illustrated in \Cref{exa:3}), every minimal
  vertex cover of $\mathcal{J}_s (G)$ can be obtained by stacking
  vertex covers of $\mathcal{J}_0 (G) \cong G$ and $\mathcal{J}_1 (G)$
  corresponding to irreducible 1-covers and 2-covers of $G$. The
  cardinality of these vertex covers is equal to the degree of the
  corresponding monomials in $\Bbbk [X]$. Based on \Cref{thm:4}, it is
  enough to show that, for a very well-covered graph $G$ with $2n$
  vertices, the monomials of type \ref{item:3} have degree $n$, and
  the monomials of types \ref{item:4} and \ref{item:5} have degree
  $2n$. This is quickly verified for types \ref{item:3} and
  \ref{item:4}, so we concentrate on monomials of type \ref{item:5}.

  Let $U$ be a nonempty independent set of $G$ such that
  \begin{itemize}
  \item $N(U)$ is not a vertex cover of $G$,
  \item $X\neq U\cup N(U)$, and
  \item the induced subgraph $G[X\setminus (U\cup N(U))]$ has no
    isolated vertex and is not bipartite.
  \end{itemize}
  We need to show that the monomial
  \begin{equation*}
    m=\prod_{v\in N(U)} v^2 \prod_{z\in X\setminus (U\cup N(U))} z
  \end{equation*}
  has degree $2n$. Since $U$ is an independent set, no two vertices in
  $U$ are adjacent, which implies that $U\cap N(U) =
  \varnothing$. It follows that
  \begin{equation*}
    |X\setminus (U\cup N(U))| = 2n - |U| - |N(U)|,
  \end{equation*}
  so the degree of $m$ is
  \begin{equation*}
    2|N(U)| + |X\setminus (U\cup N(U))| =
    2n + |N(U)| - |U|.
  \end{equation*}
  Thus, to show that $m$ has degree $2n$, we need to prove that $|N(U)|=|U|$.

  By \Cref{thm:5}, $G$ has a perfect matching $M$ that satisfies the
  property (P). Using again the fact that $U\cap N(U) = \varnothing$,
  we deduce that, for every $u\in U$, $M(u)\in N(U)$. Hence, we have
  $|N(U)|\geqslant |U|$ by definition of perfect matching.  If
  $|N(U)|>|U|$, then there is a vertex $v\in N(U)$ with match
  $M(v) \notin U$ and adjacent to some $u\in U$. Suppose for contradiction that $M(v) \in N(U)$. Then, $M(v)$ is adjacent to some
  $u'\in U$. By the property (P), $u\neq u'$, since $u$ is not
  adjacent to $M(v)$, and $u$ is adjacent to $u'$, which contradicts the
  assumption that $U$ is an independent set. Since $X\neq U\cup N(U)$,
  we deduce that $M(v) \in X\setminus (U\cup N(U))$. 
  
  Note that
  $X\setminus (U\cup N(U))$ contains at least two elements, otherwise
  $N(U)$ would be a vertex cover of $G$, which is explicitly
  forbidden.  Because the induced subgraph
  $G[X\setminus (U\cup N(U))]$ has no isolated vertex, there is a
  vertex $z\in X\setminus (U\cup N(U))$ such that $z\neq M(v)$ and
  $M(v)$ is adjacent to $z$. 
  Recall that $v$ is adjacent to $u\in U$.
  By the property (P), $u$ must also be adjacent
  to $z$, which contradicts the choice of $z\notin N(U)$. Thus, we
  cannot have that $|N(U)|>|U|$ and deduce that
  $|N(U)|=|U|$, which concludes the proof.
\end{proof}

\begin{example}
  The graph in \Cref{fig:4} is an example of a very well-covered graph
  from \cite[\S 0]{MR677051}; we refer to it as Favaron's graph
  $G_1$.

  This graph is not bipartite because it contains the triangle
  $\{c,f,h\}$. Moreover, $G_1$ has a single perfect matching
  $M= \{\{a,e\},\{b,f\},\{c,g\},\{d,h\}\}$, consisting of all the
  vertical edges in the figure.  Following \Cref{thm:4}, there is a
  unique irreducible 2-cover of type \ref{item:5}; it arises by taking
  $U=\{a\}$, so $N(U)=\{e\}$ and the corresponding monomial is
  $bcde^2fgh$.
    \begin{figure}[htb]
    \centering
    \begin{tikzpicture}
      \tikzset{vertex/.style = {shape=circle,draw,thick,inner sep=2pt}}
      \tikzset{edge/.style = {thick}}
      \tikzset{every label/.append style={font=\small}}

      \begin{scope}[xscale=2,yscale=0.75]
        
        \node[vertex] (x1) at (0,4) [label=above:$a$] {};
        \node[vertex] (x2) at (1,4) [label=above:$b$] {};
        \node[vertex] (x3) at (2,4) [label=above:$c$] {};
        \node[vertex] (x4) at (3,4) [label=above:$d$] {};
        \node[vertex] (y1) at (0,0) [label=left:$e$] {};
        \node[vertex] (y2) at (1,1) [label=left:$f$] {};
        \node[vertex] (y3) at (2,1) [label=right:$g$] {};
        \node[vertex] (y4) at (3,0) [label=right:$h$] {};

        \draw[edge] (x1) -- (y1) -- (x2) -- (y2) -- (x3) -- (y3);
        \draw[edge] (y1) -- (x3) -- (y2) -- (y4) -- (y1);
        \draw[edge] (x3) -- (y4) -- (x4);
      \end{scope}
    \end{tikzpicture}
    \caption{Favaron's graph $G_1$.}
    \label{fig:4}
  \end{figure}
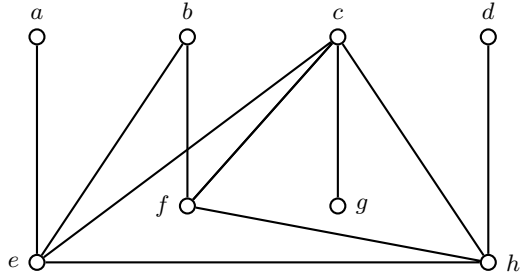
\end{example}

\section{Principal Jets of Clutters}
\label{sec:principal-jets}

In their work on jets of determinantal varieties \cite{MR3270176},
S.~Ghorpade, B.~Jonov, and B.A.~Sethuraman study a geometric object
described as ``the closure of the set of jets supported over the
smooth points of the base scheme'' and call it the \emph{principal
  component} of the jet scheme. To the best of our knowledge, this is
the first instance in the literature where the principal component is
named, although it had previously appeared in
\cite{MR2100311,MR2166800}. We define an analog of the principal
component for ideals in a polynomial ring in a way that is compatible
with this geometric notion. We avoid using the word ``component''
since the object we define is not necessarily a prime or primary
ideal.

For an ideal $I$ in the polynomial ring $\Bbbk [X]$, let $Z$ denote
the vanishing locus of $I$ in the affine space
$\mathbb{A}^{|X|}_\Bbbk$ and let $Z_\text{sing}$ denote its singular
locus. Consider the natural projection
$\pi_s \colon \mathcal{J}_s (\mathbb{A}^{|X|}_\Bbbk)\to
\mathbb{A}^{|X|}_\Bbbk$ (see \cite[Example 2.1]{MR2483946}). The
$s$-jets of $Z$ supported over the smooth locus of $Z$ are the
elements of the set
$\mathcal{J}_s (Z) \setminus \pi_s^{-1} (Z_\text{sing})$. The
polynomials vanishing on the Zariski closure of this set form an ideal
\begin{equation*}
  \mathbb{I} \left(
    \overline{\mathcal{J}_s (Z) \setminus \pi_s^{-1} (Z_\text{sing})}
  \right)\subset\Bbbk[\mathcal{J}_s(X)].
\end{equation*}

\begin{definition}
  Let $s\in \mathbb{N}$.  With the notation above, the \emph{ideal of
    principal $s$-jets of $I$} is the ideal
  \begin{equation*}
    \mathcal{P}_s (I) := 
    \mathbb{I} \left(
      \overline{\mathcal{J}_s (Z) \setminus \pi_s^{-1} (Z_\text{sing})}
    \right)\subset\Bbbk[\mathcal{J}_s(X)].
  \end{equation*}
\end{definition}

\begin{remark}\label{rem:4}
  $\mathcal{P}_s (I) = \mathcal{P}_s (\sqrt{I})$ because $I$ and
  $\sqrt{I}$ have the same vanishing locus.
\end{remark}

The next result offers an equivalent characterization of principal
jets that is more useful in computations.

\begin{proposition}\label{pro:2}
  Let $I$ be an ideal of $\Bbbk [X]$. For any $s\in \mathbb{N}$, let
  $\iota_s \colon \Bbbk [X] \to \Bbbk [ \mathcal{J}_s (X)]$ denote the
  inclusion of $\Bbbk$-algebras sending the variable $x\in X$ to
  $x_0 \in \mathcal{J}_s (X)$. There is an equality of ideals
  \begin{equation*}
    \mathcal{P}_s (I) = 
    \mathbb{I} ( \mathcal{J}_s (Z) ) :
    \iota_s (\mathbb{I} (Z_\text{sing})).
  \end{equation*}
\end{proposition}

\begin{proof}
  We have equalities
  \begin{equation*}
    \mathcal{P}_s (I)
    = \mathbb{I} \left(
      \mathcal{J}_s (Z) \setminus \pi_s^{-1} (Z_\text{sing})
    \right)
    = \mathbb{I} (\mathcal{J}_s (Z)) : \mathbb{I} (\pi_s^{-1} (Z_\text{sing})),
  \end{equation*}
  the first following from properties of the Zariski closure (see
  \cite[Chapter 4, \S 4, Lemma 3 (i)]{MR3330490}), and the second from
  properties of quotient ideals (see \cite[Chapter 4, \S 4, Exercise
  8]{MR3330490}). Finally, use the equivalence between affine
  varieties and affine $\Bbbk$-algebras sending $\pi_s$ to $\iota_s$
  (see, for example, \cite[Chapter 1, Corollary 3.8]{MR0463157}) or a
  direct computation to conclude that
  $\mathbb{I} (\pi_s^{-1} (Z_\text{sing})) = \iota_s (\mathbb{I}
  (Z_\text{sing}))$.
\end{proof}

\begin{remark}
  The \texttt{Jets} package \cite{MR4538337,JetsSource} for the
  software Macaulay2 \cite{M2} includes a \texttt{principalComponent}
  method to compute principal jets which uses a description similar to
  the one offered in \Cref{pro:2}.
\end{remark}

We proceed to study principal jets of monomial ideals, with the ultimate goal to describe their irreducible components. Based on
\Cref{rem:4}, it is enough to consider squarefree monomial ideals,
i.e., edge ideals of clutters. The case of edge ideals of graphs was
studied by N.~Iammarino \cite[\S 3]{2201.04164}.

\begin{proposition}\label{pro:3}
  Let $C$ be a clutter with vertex set $X$, and let $W_1,\dots,W_n$ be
  the minimal vertex covers of $C$. For every $s\in\mathbb{N}$,
  \begin{equation*}
    \mathcal{P}_s (I(C)) =
    I (\mathcal{J}_s (C)) : \bigcap_{1\leqslant i<j\leqslant n}
    \langle \mathcal{J}_0 (W_i \cup W_j)\rangle.
  \end{equation*}
  In particular, $\mathcal{P}_s (I(C))$ is a squarefree monomial
  ideal.
\end{proposition}

\begin{proof}
  Let $Z$ be the vanishing locus of the edge ideal $I(C)$.
  By \Cref{def:2}, we have
  \begin{equation*}
    \mathbb{I} (\mathcal{J}_s (Z)) =
    \sqrt{ \mathcal{J}_s (I(C)) } =
    I (\mathcal{J}_s (C)).
  \end{equation*}
  Observe that $Z$ decomposes as the union of the linear spaces
  defined by the ideals
  $\langle W_1\rangle, \dots, \langle W_n\rangle$. Hence,
  $Z_\text{sing}$ consists of those points that belong to the
  intersection of at least two components of $Z$. For
  $1\leqslant i<j\leqslant n$, the intersection of the components
  $\langle W_i\rangle$ and $\langle W_j\rangle$ is defined by the
  ideal
  $\langle W_i\rangle +\langle W_j\rangle = \langle W_i \cup
  W_j\rangle$. Therefore, we have
  \begin{equation*}
    \mathbb{I} (Z_\text{sing}) =
    \bigcap_{1\leqslant i<j\leqslant n}
    \langle W_i \cup W_j\rangle.
  \end{equation*}
  The description of $\mathcal{P}_s (I(C))$ follows from \Cref{pro:2}.

  Since intersections and quotients of squarefree monomial ideals are
  squarefree monomial ideals, it follows that $\mathcal{P}_s (I(C))$
  is also a squarefree monomial ideal.
\end{proof}

In order to prove our main result, we need
the following technical lemma. The proof is inspired by \cite[Discussion
  4.11]{MR3779569}. Recall that the \emph{support} of a monomial $m$
in a polynomial ring, denoted $\operatorname{supp} (m)$, is the set of
variables dividing $m$.

\begin{lemma}\label{lem:2}
  Let $I$ be a squarefree monomial ideal in the polynomial ring
  $R_n = \Bbbk [x_1,\dots,x_n]$, and let $s$ be any positive integer. If
  $M$ is a minimal monomial generator of $I^{(s)}$, then the support
  of $M$ is the union of the supports of some minimal monomial
  generators of $I$.
\end{lemma}

\begin{proof}
  If $I=\langle 1\rangle$, then $I^{(s)} = \langle 1\rangle$ and the
  result is true.

  Assume $I$ is proper and proceed by induction on $n$. If $n=1$, then
  $I=\langle x_1\rangle$. Hence, we have
  $I^{(s)} = I^s = \langle x_1^s\rangle$ and the result is true.

  Now consider an integer $n\geqslant 1$ and assume the result is true
  for squarefree monomial ideals in $R_n$. Let $I$ be a squarefree
  monomial ideal in $R_{n+1}$. Let $J$ be the ideal of $R_{n+1}$ generated by
  all monomials in $I$ not divisible by $x_{n+1}$, and let
  $K = I:\langle x_{n+1}\rangle$. The ideals $J$ and $K$ are
  squarefree monomial ideals whose minimal monomial generators are not
  divisible by $x_{n+1}$. Moreover, we have
  $I = (J+\langle x_{n+1}\rangle) \cap K$, which implies
  \begin{equation*}
    I^{(s)} = (J+\langle x_{n+1}\rangle)^{(s)} \cap K^{(s)}
  \end{equation*}
  for every positive integer $s$. Using \cite[Theorem 7.8]{MR3566223}
  or the more general \cite[Theorem 3.4]{MR4074049}, we deduce that
  \begin{equation*}
    (J+\langle x_{n+1}\rangle)^{(s)} =
    \sum_{i=0}^s J^{(s-i)} \langle x_{n+1}\rangle^{(i)}
    = \sum_{i=0}^s J^{(s-i)} \langle x_{n+1}^i\rangle.
  \end{equation*}
  Altogether, for every positive integer $s$, we get an equality
  \begin{equation}\label{eq:3}
    I^{(s)} =
    \left( \sum_{i=0}^s J^{(s-i)} \langle x_{n+1}^i\rangle \right)
    \cap K^{(s)}.
  \end{equation}
  Now, let $M$ be a minimal monomial generator of $I^{(s)}$.
  We have the following two cases.

  \begin{description}
  \item[Case 1] $M$ is not divisible by $x_{n+1}$. Then, \Cref{eq:3}
    implies that $M$ is divisible by a minimal monomial generator $M'$
    of $J^{(s)}$. By definition, we have $J\subseteq K$. Hence,
    $M'\in J^{(s)}\subseteq K^{(s)}$, and so $M'\in I^{(s)}$. We deduce that $M=M'$, as otherwise
    $M$ would not be a minimal generator of $I^{(s)}$.  Since $M$ is
    also a minimal monomial generator of the ideal
    $J^{(s)} \cap R_n = (J \cap R_n)^{(s)}$ of $R_n$, the inductive
    hypothesis allows us to conclude that
    \begin{equation*}
      \operatorname{supp} (M) =
      \bigcup_{j=1}^p \operatorname{supp} (f_j)
    \end{equation*}
    for some minimal monomial generators $f_1,\dots,f_p$ of $J$.
  \item[Case 2] $M$ is divisible by $x_{n+1}$. Then, \Cref{eq:3}
    implies that $M$ is divisible by:
    \begin{itemize}
    \item $M_1 x_{n+1}^h$, for some integer $h\in \{1,\dots,s\}$ and a
      minimal monomial generator $M_1$ of $J^{(s-h)}$;
    \item a minimal monomial generator $M_2$ of $K^{(s)}$.
    \end{itemize}
    Using the inductive hypothesis as in the previous case, we obtain
    that
    \begin{equation*}
      \operatorname{supp} (M_1) =
      \bigcup_{j=1}^p \operatorname{supp} (f_j)
    \end{equation*}
    for some minimal monomial generators $f_1,\dots,f_p$ of $J$, and
    \begin{equation*}
      \operatorname{supp} (M_2) =
      \bigcup_{k=1}^q \operatorname{supp} (g_k)
    \end{equation*}
    for some minimal monomial generators $g_1,\dots,g_q$ of $K$. Given
    the minimality of $M$, we deduce that
    \begin{equation*}
      \begin{split}
        \operatorname{supp} (M)
        &= \operatorname{supp} (M_1 x_{n+1}^h) \cup \operatorname{supp} (M_2) \\
        &= \operatorname{supp} (M_1) \cup \{x_{n+1}\} \cup \operatorname{supp} (M_2) \\
        &= \bigcup_{j=1}^p \operatorname{supp} (f_j) \cup \{x_{n+1}\} \cup
          \bigcup_{k=1}^q \operatorname{supp} (g_k)\\
        &= \bigcup_{j=1}^p \operatorname{supp} (f_j) \cup
          \bigcup_{k=1}^q \operatorname{supp} (g_k x_{n+1}).
      \end{split}
    \end{equation*}
  \end{description}
  Combining the two cases, for every minimal monomial generator $M$ of
  $I^{(s)}$ we have
  \begin{equation*}
    \operatorname{supp} (M) =
    \bigcup_{j=1}^p \operatorname{supp} (f_j) \cup
    \bigcup_{k=1}^q \operatorname{supp} (g_k x_{n+1}),
  \end{equation*}
  for some minimal monomial generators $f_1,\dots,f_p$ of $J$ and some
  minimal monomial generators $g_1,\dots,g_q$ of $K$, where $q=0$ when
  $M$ is not divisible by $x_{n+1}$. To conclude, observe that
  $I = J+K\langle x_{n+1}\rangle$ and, therefore, a minimal monomial
  generator of $I$ is a minimal monomial generator $f$ of $J$ or it is
  equal to $g x_{n+1}$ where $g$ is minimal monomial generator of
  $K$. This shows that the result is true in $R_{n+1}$.
\end{proof}

We come to the main result of this section on the irreducible components of the principal jets of a monomial ideal.

\begin{theorem}\label{thm:7}
  Let $C$ be a clutter with vertex set $X$, and let $W_1,\dots,W_n$ be
  the minimal vertex covers of $C$. For every $s\in\mathbb{N}$,
  \begin{equation*}
    \mathcal{P}_s (I(C)) =
    \bigcap_{k=1}^n \langle \mathcal{J}_s (W_k) \rangle
  \end{equation*}
  is the irredundant primary decomposition of $\mathcal{P}_s (I(C))$.
\end{theorem}

\begin{proof}
  Since $I (\mathcal{J}_s (C))$ is a squarefree monomial ideal, and
  therefore radical, it is equal to the intersection of its minimal
  primes $\mathfrak{p}_1,\dots,\mathfrak{p}_r$. By \Cref{pro:3}, we
  have
  \begin{equation*}
    \begin{split}
      \mathcal{P}_s (I(C))
      &= I (\mathcal{J}_s (C)) : \bigcap_{1\leqslant i<j\leqslant n}
      \langle \mathcal{J}_0 (W_i \cup W_j)\rangle\\
      &= \bigcap_{h=1}^r \left( \mathfrak{p}_h : \bigcap_{1\leqslant i<j\leqslant n}
      \langle \mathcal{J}_0 (W_i \cup W_j)\rangle \right),
    \end{split}
  \end{equation*}
  where the second equality follows from properties of quotient ideals
  (see \cite[Chapter 4, \S 4, Exercise 17]{MR3330490}). We claim that
  the quotients in the intersection above are either equal to the
  whole ring or to one of the ideals
  $\langle \mathcal{J}_s (W_k) \rangle$, and that all such ideals
  appear in the intersection. Since the ideals
  $\langle \mathcal{J}_s (W_k) \rangle$ are prime and incomparable
  under inclusion, the result would then follow.
  
  If $\mathfrak{p}$ is a minimal prime of $I(\mathcal{J}_s (C))$, then
  there is a minimal vertex cover $V$ of $\mathcal{J}_s (C)$ such that
  $\mathfrak{p} = \langle V\rangle$. Equivalently, if $M$ is the
  product of the elements in $V$, then $M$ is a minimal generator of
  $I_c (\mathcal{J}_s (C))$, the ideal of covers of $\mathcal{J}_s (C)$. By
  \Cref{thm:3}, the monomial $\delta_s (M)$ is a minimal generator of
  $I_c (C)^{(s+1)}$, where
  $\delta_s \colon \Bbbk [\mathcal{J}_s (X)] \to \Bbbk [X]$ is the
  depolarization map. By \Cref{lem:2}, the support of $\delta_s (M)$
  is the union of the supports of some minimal monomial generators
  $m_1,\dots,m_t$ of $I_c (C)$. In turn, the support of a minimal
  monomial generator $m_j$ of $I_c (C)$ is a minimal vertex cover $W_{i_j}$ of
  $C$. Altogether, we have
  \begin{equation}\label{eq:4}
    \delta_s (V) = \delta_s ( \operatorname{supp} (M)) =
    \operatorname{supp} (\delta_s (M)) =
    \bigcup_{j=1}^t \operatorname{supp} (m_j) =
    \bigcup_{j=1}^t W_{i_j}
  \end{equation}
  for some $i_1,\dots,i_t \in \{1,\dots,n\}$. We now consider the following two cases, depending on whether $t=1$ or $t>1$.
  \begin{description}
  \item[Case 1] $\delta_s (V) = W_k$ for some $k\in \{1,\dots,n\}$.
    This implies
    $V\subseteq \delta_s^{-1} (W_k) = \mathcal{J}_s (W_k)$. Since
    $\mathcal{J}_s (W_k)$ is a minimal vertex cover of
    $\mathcal{J}_s (C)$ by \Cref{thm:2}, we deduce that
    $V= \mathcal{J}_s (W_k)$ or, equivalently,
    $\mathfrak{p} = \langle \mathcal{J}_s (W_k)\rangle$.

    The minimality of the vertex covers $W_1,\dots,W_n$ implies that,
    for any pair of distinct indices $i,j\in \{1,\dots,n\}$,
    $(W_i \cup W_j) \setminus W_k \neq \varnothing$. Hence, we can
    introduce the monomial
    \begin{equation*}
      F_k = \prod_{x\in X\setminus W_k} x_0 \in \bigcap_{1\leqslant i<j\leqslant n}
      \langle \mathcal{J}_0 (W_i \cup W_j)\rangle.
    \end{equation*}
    Since $\mathfrak{p}$ is prime, the assumption $\delta_s (V) = W_k$
    implies that $F_k \notin \mathfrak{p} = \langle V\rangle$. If $m$
    is any monomial in $\mathfrak{p} : \langle F_k\rangle$, then
    $mF_k \in \mathfrak{p}$ and, therefore,
    $m\in\mathfrak{p}$. Because $\mathfrak{p} : \langle F_k\rangle$ is
    a monomial ideal, we deduce that
    $\mathfrak{p} : \langle F_k\rangle \subseteq \mathfrak{p}$. Using
    basic properties of quotient ideals, we conclude that
    \begin{equation*}
      \mathfrak{p} \subseteq
      \mathfrak{p} : \bigcap_{1\leqslant i<j\leqslant n}
      \langle \mathcal{J}_0 (W_i \cup W_j)\rangle
      \subseteq \mathfrak{p} : \langle F_k \rangle
      \subseteq \mathfrak{p},
    \end{equation*}
    and therefore
    \begin{equation*}
      \mathfrak{p} : \bigcap_{1\leqslant i<j\leqslant n}
      \langle \mathcal{J}_0 (W_i \cup W_j)\rangle
      = \mathfrak{p}
      = \langle \mathcal{J}_s (W_k)\rangle.
    \end{equation*}
    Notice also that by taking $V=\mathcal{J}_s (W_k)$ we can
    guarantee that $\langle \mathcal{J}_s (W_k)\rangle$ will appear
    among the primes in the decomposition of $\mathcal{P}_s (I(C))$.
  \item[Case 2] $\delta_s (V) \supseteq W_{k}\cup W_l$ for some
    $k,l\in \{1,\dots,n\}$ with $k\neq l$. This implies
    $V\supseteq \mathcal{J}_0 (W_k\cup W_l)$ or, equivalently,
    $\mathfrak{p} \supseteq \langle \mathcal{J}_0 (W_{k}\cup W_{l})
    \rangle$. We conclude that
    \begin{equation*}
      \mathfrak{p} : \bigcap_{1\leqslant i<j\leqslant n}
      \langle \mathcal{J}_0 (W_i \cup W_j)\rangle
      \supseteq \mathfrak{p} :
      \langle \mathcal{J}_0 (W_{k} \cup W_{l})\rangle
      = \Bbbk [\mathcal{J}_s (X)].
    \end{equation*}
  \end{description}
  Since \Cref{eq:4} forces one of the two cases above to occur, the
  claim in the opening paragraph of the proof holds true.
\end{proof}

An explicit description for the generators of the principal jets of
the edge ideal of a clutter is now within easy reach.

\begin{corollary}\label{cor:2}
  Let $C$ be a clutter with vertex set $X$. For every
  $s\in\mathbb{N}$,
  \begin{equation*}
    \mathcal{P}_s (I(C)) =
    \langle x_{1,i_1} \cdots x_{r,i_r} \colon \{x_1,\dots,x_r\} \in E(C),
    0\leqslant i_1,\dots,i_r\leqslant s\rangle,
  \end{equation*}
  and this generating set is minimal.
\end{corollary}

\begin{proof}
  Let $W_1,\dots,W_n$ be the minimal vertex covers of $C$, and
  consider a monomial $m\in \mathcal{P}_s (I(C))$. Applying the
  depolarization map $\delta_s$ to the decomposition of
  $\mathcal{P}_s (I(C))$ in \Cref{thm:7}, we get
  \begin{equation*}
    \delta_s (m) \in
    \bigcap_{k=1}^n \delta_s (\langle \mathcal{J}_s (W_k) \rangle)
    \subseteq \bigcap_{k=1}^n \langle W_k \rangle = I(C).
  \end{equation*}
  Hence, there is an edge $\{x_1,\dots,x_r\}$ of $C$ such that $x_1\cdots x_r$ divides $\delta_s(m)$. This implies that $x_{1,i_1}\cdots x_{r,i_r}$ divides $m$ for some integers
  $i_1,\dots,i_r \in \{0,\dots,s\}$. On the other hand, a minimal
  vertex cover $W_k$ of $C$ intersects the edge $\{x_1,\dots,x_r\}$ in
  some vertex, say $x_j$. Thus, we have
  $x_{j,i_j}\in \mathcal{J}_s (W_k)$ for every
  $i_j \in \{0,\dots,s\}$, which allows us to conclude that
  $x_{1,i_1}\cdots x_{r,i_r} \in \mathcal{P}_s (I(C))$ for any choice of
  indices $i_1,\dots,i_r \in \{0,\dots,s\}$. We conclude that the
  monomials $x_{1,i_1}\cdots x_{r,i_r}$ such that
  $\{x_1,\dots,x_r\} \in E(C)$ and $i_1,\dots,i_r \in \{0,\dots,s\}$
  are the minimal generators of $\mathcal{P}_s (I(C))$.
\end{proof}

In light of this result, we introduce the following.

\begin{definition}
  Let $C$ be a clutter with vertex set $X$. For $s\in \mathbb{N}$, the
  \emph{clutter of principal $s$-jets of $C$}, denoted
  $\mathcal{P}_s (C)$, is the clutter with vertex set
  $\mathcal{J}_s (X)$ whose edge ideal is given by
  $\mathcal{P}_s (I(C))$.
\end{definition}

\begin{example}\label{exa:4}
  Consider the graph $G$ with vertex set $\{x,y,z\}$ whose edge ideal
  is given by $I(G) = \langle xy,yz\rangle$ (the same ideal as in
  \Cref{exa:1}). By \Cref{cor:2}, we have
  \begin{equation*}
    \begin{split}
      \mathcal{P}_1 (I(G)) = \langle &x_0 y_0,y_0 z_0, x_0 y_1,x_1 y_0,
                          y_0 z_1,y_1 z_0, x_1 y_1, y_1 z_1 \rangle,\\
      \mathcal{P}_2 (I(G)) = \langle &x_0 y_0,y_0 z_0, x_0 y_1,x_1 y_0,
                             y_0 z_1,y_1 z_0,x_0 y_2,x_1 y_1,x_2 y_0,\\
                          &y_0 z_2,y_1 z_1, y_2 z_0,  x_1 y_2, x_2 y_1,
                          y_1 z_2, y_2 z_1, x_2 y_2, y_2 z_2\rangle.
    \end{split}
  \end{equation*}
  The graphs $G, \mathcal{P}_1 (G)$, and $\mathcal{P}_2 (G)$ are
  pictured in \Cref{fig:5}.
  Moreover, since $I(G) = \langle x,z\rangle \cap \langle y\rangle$,
  we have the following decompositions:
  \begin{equation*}
    \begin{split}
      \mathcal{P}_1 (I(G)) = \langle
      &x_0, z_0, x_1, z_1\rangle
        \cap \langle y_0,y_1\rangle,\\
      \mathcal{P}_2 (I(G)) = \langle
      &x_0, z_0, x_1, z_1, x_2, z_2\rangle
        \cap \langle y_0,y_1,y_2\rangle.
    \end{split}
  \end{equation*}
    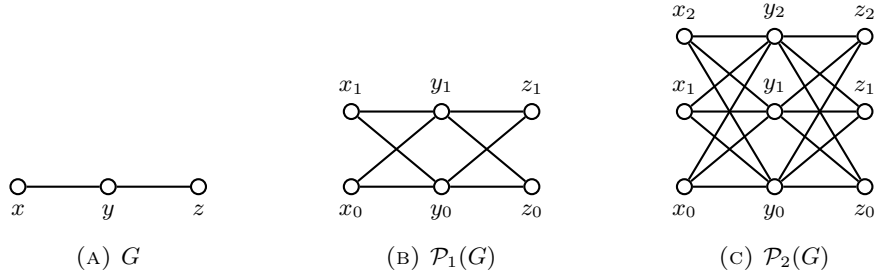
\begin{figure}[htb]
    \centering
    \begin{subfigure}[b]{0.3\textwidth}
      \centering
      \begin{tikzpicture}[xscale=1.2]
        \tikzset{vertex/.style = {shape=circle,draw,thick,inner sep=2pt}}
        \tikzset{edge/.style = {thick}}
        \tikzset{every label/.append style={font=\small}}

        \node[vertex] (x) at (1,0) [label=below:{$x$}] {};
        \node[vertex] (y) at (2,0) [label=below:{$y$}] {};
        \node[vertex] (z) at (3,0) [label=below:{$z$}] {};

        \draw[edge] (x) -- (y) -- (z);
      \end{tikzpicture}
      \caption{$G$}
    \end{subfigure}
    \hfill
    \begin{subfigure}[b]{0.3\textwidth}
      \centering
      \begin{tikzpicture}[xscale=1.2]
        \tikzset{vertex/.style = {shape=circle,draw,thick,inner sep=2pt}}
        \tikzset{edge/.style = {thick}}
        \tikzset{every label/.append style={font=\small}}

        \node[vertex] (x0) at (1,0) [label=below:{$x_0$}] {};
        \node[vertex] (y0) at (2,0) [label=below:{$y_0$}] {};
        \node[vertex] (z0) at (3,0) [label=below:{$z_0$}] {};
        \node[vertex] (x1) at (1,1) [label=above:{$x_1$}] {};
        \node[vertex] (y1) at (2,1) [label=above:{$y_1$}] {};
        \node[vertex] (z1) at (3,1) [label=above:{$z_1$}] {};

        \draw[edge] (x0) -- (y0) -- (z0);
        \draw[edge] (x1) -- (y0) -- (z1);
        \draw[edge] (x0) -- (y1) -- (z0);
        \draw[edge] (x1) -- (y1) -- (z1);
      \end{tikzpicture}
      \caption{$\mathcal{P}_1 (G)$}
    \end{subfigure}
    \hfill
    \begin{subfigure}[b]{0.3\textwidth}
      \centering
      \begin{tikzpicture}[xscale=1.2]
        \tikzset{vertex/.style = {shape=circle,draw,thick,inner sep=2pt}}
        \tikzset{edge/.style = {thick}}
        \tikzset{every label/.append style={font=\small}}

        \node[vertex] (x0) at (1,0) [label=below:{$x_0$}] {};
        \node[vertex] (y0) at (2,0) [label=below:{$y_0$}] {};
        \node[vertex] (z0) at (3,0) [label=below:{$z_0$}] {};
        \node[vertex] (x1) at (1,1) [label=above:{$x_1$}] {};
        \node[vertex] (y1) at (2,1) [label=above:{$y_1$}] {};
        \node[vertex] (z1) at (3,1) [label=above:{$z_1$}] {};
        \node[vertex] (x2) at (1,2) [label=above:{$x_2$}] {};
        \node[vertex] (y2) at (2,2) [label=above:{$y_2$}] {};
        \node[vertex] (z2) at (3,2) [label=above:{$z_2$}] {};

        \draw[edge] (x0) -- (y0) -- (z0);
        \draw[edge] (x1) -- (y0) -- (z1);
        \draw[edge] (x0) -- (y1) -- (z0);
        \draw[edge] (x1) -- (y1) -- (z1);
        \draw[edge] (x2) -- (y0) -- (z2);
        \draw[edge] (x0) -- (y2) -- (z0);
        \draw[edge] (x1) -- (y2) -- (z1);
        \draw[edge] (x2) -- (y1) -- (z2);
        \draw[edge] (x2) -- (y2) -- (z2);
      \end{tikzpicture}
      \caption{$\mathcal{P}_2 (G)$}
    \end{subfigure}
    \caption{A path graph, and its first and second principal jets.}
    \label{fig:5}
  \end{figure}
\end{example}

\begin{remark}
  For a graph $G$, the graph $\mathcal{P}_s (G)$ is isomorphic to the
  \emph{disjunction} (also known as \emph{disjunctive product} or
  \emph{co-normal product}) of $G$ and a graph with $s+1$ vertices and
  no edges (see, for example, \cite[Definition 10]{MR2404539} and
  references therein).  We are grateful to Jack Huizenga and Zach
  Teitler for mentioning this connection with the existing literature.
\end{remark}

As is the case in the geometric setting, principal jets of clutters
are better behaved than jets. The following result, which is an
immediate consequence of \Cref{thm:7}, illustrates this behavior by
completely describing the minimal vertex covers of the principal jets
of a clutter.

\begin{corollary}
  Let $C$ be a clutter, and let $W_1,\dots,W_n$ be the minimal vertex
  covers of $C$. For every $s\in\mathbb{N}$, the sets
  $\mathcal{J}_s (W_1),\dots,\mathcal{J}_s (W_n)$ are the minimal
  vertex covers of the clutter of principal $s$-jets
  $\mathcal{P}_s (C)$. In particular, $C$ is (very) well-covered if
  and only if, for every $s\in \mathbb{N}$, $\mathcal{P}_s (C)$ is
  (very) well-covered.
\end{corollary}

\section{Invariants of Principal Jets of Clutters}
\label{sec:invar-princ-jets}

In this section, we study Hilbert series and other invariants of principal jets of squarefree monomial ideals. This is inspired by S. Ghorpade, B. Jonov, and B.A. Sethuraman's result on a simple formula for the Hilbert series of
the first order principal component of the $2\times 2$ minors of a
generic matrix \cite[Theorem
18]{MR3270176}.

We begin with a technical result on the depolarization map $\delta_s$ (\Cref{def:3}) that will later be used to express our formulas.

\begin{lemma}\label{lem:3}
  Let $X$ be a finite set. For every $s\in\mathbb{N}$ and every subset
  $U\subseteq X$, the number of subsets $V\subseteq \mathcal{J}_s (X)$
  of cardinality $k$ such that $\delta_s (V) = U$ depends only on the
  cardinality $j$ of $U$ and is given by
  \begin{equation}\label{eq:5}
    \ell_s (j,k) =
    \sum_{\substack{a\in \mathbb{Z}_+^j,\\|a| = k}}
    \prod_{h=1}^j \binom{s+1}{a_h},
  \end{equation}
  where $\mathbb{Z}_+$ denotes the set of positive integers, and
  $|a|= a_1+\dots+a_j$ for every
  $a = (a_1,\dots,a_j) \in \mathbb{Z}_+^j$.
\end{lemma}

\begin{proof}
  If $U=\varnothing$, then the only subset
  $V\subseteq \mathcal{J}_s (X)$ with $\delta_s (V) = \varnothing$ is
  $V=\varnothing$, so $\ell_s (0,0) = 1$. In this case, the
  summation in \Cref{eq:5} is indexed by the singleton
  $\mathbb{Z}_+^0$, and the only summand is an empty product which is
  equal to 1. Therefore, the formula holds when $U=\varnothing$.

  Now, consider a nonempty subset $U = \{x_1,\dots,x_j\} \subseteq X$
  and a subset $V\subseteq \mathcal{J}_s (X)$.  For every
  $h\in\{1,\dots,j\}$, let $a_h = |\delta_s^{-1} (\{x_h\}) \cap V|$,
  i.e., the number of preimages of $x_h$ contained in $V$.  If $V$ has
  cardinality $k$ and satisfies $\delta_s (V) = U$, then
  $a_1,\dots,a_j$ are positive and add up to $k$, i.e.,
  $a \in \mathbb{Z}_+^j$ and $|a| = k$. Since
  $|\delta_s^{-1} (\{x_h\})|=s+1$, there are $\binom{s+1}{a_h}$
  different ways to select the elements of
  $\delta_s^{-1} (\{x_h\}) \cap V$. Given that preimages can be
  selected independently for each $x_h$, we multiply these binomial
  coefficients to count the number of subsets $V$ corresponding to a
  particular choice of $a \in \mathbb{Z}_+^j$ with $|a| = k$. Note
  that a subset $V$ cannot correspond to two different
  $a$'s. Therefore, the total number of subsets $V$ is obtained by
  adding the products of binomial coefficients over the different
  possible $a$'s.
\end{proof}

The function
$\ell_s \colon \mathbb{N} \times \mathbb{N} \to \mathbb{N}$ counts
the number $\ell_s(j,k)$ of ways a subset $U\subseteq X$ of cardinality $j$ lifts to
a subset $V\subseteq \mathcal{J}_s (X)$ of cardinality $k$ along the
restriction $\mathcal{J}_s (X) \to X$ of the depolarization map
$\delta_s$. We will enter the values of $\ell_s (j,k)$
in a matrix $L_s$ with row index $j\geqslant 0$ and column index
$k\geqslant 0$ (see \Cref{exa:5} and \Cref{exa:6}). We refer to $\ell_s$
as the \emph{order $s$ lifting function}, and to $L_s$ as the
\emph{order $s$ lifting matrix}. Note that $\ell_s$ does not depend on
$|X|$.

\begin{remark}
  The lifting function $\ell_s$ can be computed by setting
  \begin{equation*}
    \ell_s (j,k) =
    \begin{cases}
      0, &\text{if } k<j \text{ or } k>(s+1)j,\\
      1, &\text{if } k=(s+1)j,\\
      (s+1)^j, &\text{if } k=j,      
    \end{cases}
  \end{equation*}
  and then using the recursive formula
  \begin{equation*}
    \ell_s (j,k) = \sum_{i=1}^{\min \{k,s+1\}}
    \binom{s+1}{i} \ell_s (j-1,k-i).
  \end{equation*}
  This approach may be more efficient than using the explicit formula
  of \Cref{lem:3}. We have implemented methods to compute lifting
  functions and lifting matrices in version 1.2 of the Jets package
  for Macaulay2.
\end{remark}

Recall that, for a homogeneous ideal $I$ in a standard graded
polynomial ring $\Bbbk [X]$, the \emph{Hilbert series} of
$\Bbbk [X]/ I$ is defined as
\begin{equation*}
  \sum_{i\in \mathbb{Z}} \dim_\Bbbk (\Bbbk [X] / I)_i \, t^i,
\end{equation*}
where the coefficient of $t^i$ is the dimension of the graded
component of degree $i$ in $\Bbbk [X] / I$ as a vector space over
$\Bbbk$. When $I \subseteq \Bbbk [X]$ is a squarefree monomial ideal,
one may associate to $I$ a simplicial complex $\Delta$ with vertex set
$X$ by taking as faces the supports of all the squarefree monomials
not contained in $I$. The ideal $I$ is called the
\emph{Stanley--Reisner ideal} of $\Delta$, and we set
$\Bbbk [\Delta] = \Bbbk [X] / I$.  The \emph{$f$-vector}
$f(\Delta) = (f_{-1} (\Delta), f_0 (\Delta), \dots, f_i (\Delta),
\dots)$ of $\Delta$ has components $f_i (\Delta)$ counting the number
of $i$-dimensional faces in $\Delta$, where the dimension of a face is
one less than its cardinality. We note that $f_{-1} (\Delta) = 1$,
since the empty set $\varnothing$ is the only face of dimension
$-1$. There is a well-known connection between the Hilbert series of
$\Bbbk [\Delta]$ and the $f$-vector of $\Delta$: if
$d = \dim (\Delta) + 1$, then the Hilbert series of $\Bbbk [\Delta]$
is given by
\begin{equation*}
  \frac{1}{(1-t)^d} \sum_{i=0}^d f_{i-1} (\Delta) t^i (1-t)^{d-i}.
\end{equation*}
Other algebraic invariants of interest can be conveniently read off
the combinatorial data. For example, $d$ is the \emph{Krull dimension}
of $\Bbbk [\Delta]$, and $f_{d-1} (\Delta)$ is the \emph{multiplicity}
of $\Bbbk [\Delta]$ (or the \emph{degree} of the projective variety
defined by $I$).  We refer the reader to \cite[\S 5]{MR1251956} or
\cite[\S 1]{MR2110098} for Stanley--Reisner ideals and $f$-vectors,
\cite[Appendix]{MR1251956} for a review of Krull dimension, and
\cite[\S 4.1]{MR1251956} for Hilbert series and multiplicity.

The $f$-vector of a finite dimensional simplicial complex has only a
finite number of nonzero entries but, for our next statement, we find
it more convenient to extend such an $f$-vector with infinitely many
zeros.

\begin{theorem}\label{thm:8}
  Let $I \subseteq \Bbbk [X]$ be the Stanley--Reisner ideal of a
  simplicial complex $\Delta$ with vertex set $X$ and $f$-vector
  $f(\Delta)$. For $s\in \mathbb{N}$, the ideal $\mathcal{P}_s (I)$ of
  principal $s$-jets of $I$ is the Stanley--Reisner ideal of a
  simplicial complex $\Gamma_s$ with vertex set $\mathcal{J}_s (X)$
  and $f$-vector
  \begin{equation*}
    f(\Gamma_s) = f(\Delta) L_s,
  \end{equation*}
  where $L_s$ is the order $s$ lifting matrix.  In particular,
  \begin{itemize}
  \item if $\Bbbk [X]/I$ has Krull dimension $d$, then
    $\Bbbk [\mathcal{J}_s(X)]/\mathcal{P}_s(I)$ has Krull dimension $(s+1)d$;
  \item $\Bbbk [\mathcal{J}_s(X)]/\mathcal{P}_s(I)$ has the same multiplicity
    $f_{d-1} (\Delta)$ as $\Bbbk [X]/I$.
  \end{itemize}
\end{theorem}

\begin{proof}
  A subset $V\subseteq \mathcal{J}_s (X)$ is a $k$-dimensional face of
  $\Gamma_s$ if and only if $V$ is the support of a squarefree
  monomial not contained in $\mathcal{P}_s (I)$. By \Cref{cor:2}, this
  is the case if and only if $\delta_s (V)$ is the support of a
  squarefree monomial not contained in $I$, i.e., if and only if
  $\delta_s (V)$ is a face of $\Delta$ of dimension at most
  $k$. Hence, we can count faces of $\Gamma_s$ by counting lifts of
  faces of $\Delta$.

  Let $U$ be a $j$-dimensional face of $\Delta$, i.e., a subset
  $U\subseteq X$ of cardinality $j+1$. \Cref{lem:3} gives the number
  $\ell_s (j+1,k+1)$ of ways $U$ lifts to a subset
  $V\subseteq \mathcal{J}_s (X)$ of cardinality $k+1$ along the
  depolarization map, i.e., the number of ways $U$ lifts to a
  $k$-dimensional face of $\Gamma_s$. Since different subsets $U$
  necessarily have different lifts, we deduce that the number of
  $k$-dimensional faces of $\Gamma_s$ that are lifts of some
  $j$-dimensional face of $\Delta$ is $f_j (\Delta) \ell_s
  (j+1,k+1)$. Therefore, the number of $k$-dimensional faces of
  $\Gamma_s$ is
  \begin{equation}\label{eq:6}
    f_k (\Gamma_s) = \sum_{j\geqslant -1} f_j (\Delta) \, \ell_s (j+1,k+1),
  \end{equation}
  where the sum is finite because $\ell_s (j+1,k+1) = 0$ when $j>k$. The
  right hand side of \Cref{eq:6} is the product of $f (\Delta)$ as a
  row vector with the $k+1$ column of the lifting matrix
  $L_s$. Therefore, $f (\Gamma_s) = f(\Delta) L_s$ as stated.

  If $\Bbbk [\Delta]$ has Krull dimension $d$, then $\Delta$ is a
  simplicial complex of dimension $d-1$. If $U$ is a face of $\Delta$
  of dimension $d-1$, then $\mathcal{J}_s (U)$ is the only lift of $U$
  to a face of $\Gamma_s$ of dimension $(s+1)d-1$. Moreover, every
  face of $\Delta$ of dimension $j-1$ lifts to a face of $\Gamma_s$ of
  dimension at most $(s+1)j-1$, so $\Gamma_s$ has no face of dimension
  $(s+1)d$ or higher. Therefore, $\Gamma_s$ has dimension $(s+1)d-1$,
  and $\Bbbk [\Gamma_s]$ has Krull dimension $(s+1)d$, as claimed.

  Finally, observe that $f_j (\Delta) = 0$ for $j\geqslant d$, while
  \begin{equation*}
    \ell_s (j+1,(s+1)d) =
    \begin{cases}
      0, &j<d-1,\\
      1, &j=d-1.
    \end{cases}
  \end{equation*}
  Hence, \Cref{eq:6} implies that the multiplicity of
  $\Bbbk [\Gamma_s]$ is $f_{(s+1)d-1} (\Gamma_s) = f_{d-1} (\Delta)$.
\end{proof}

\begin{remark}
  It is already known that the $s$-jets of an affine variety of
  dimension $d$ form an affine variety of dimension $(s+1)d$ (see
  \cite[Corollary 2.11]{MR2483946}). Since principal jets are
  supported on the smooth locus of the variety, they also have
  dimension $(s+1)d$. The proof of \Cref{thm:8} offers an alternative
  for monomial ideals.
\end{remark}

\begin{example}\label{exa:5}
  Consider the ideal $I=\langle vwx,xy,yz\rangle$ in the standard
  graded polynomial ring $\Bbbk [v,w,x,y,z]$. This is the
  Stanley--Reisner ideal of a simplicial complex $\Delta$ with facets
  (i.e., maximal faces) $\{v,w,y\},\{v,w,z\},\{v,x,z\}$, and
  $\{w,x,z\}$, which is pictured in \Cref{fig:6}.
  \begin{figure}[htb]
    \centering
    \begin{tikzpicture}[scale=1.25]
      \tikzset{vertex/.style = {shape=circle,draw,thick,inner sep=2pt,fill=white}}
      \tikzset{edge/.style = {thick}}
      \tikzset{every label/.append style={font=\small}}

      \draw[thick,fill=lightgray] (180:1) -- (60:1) -- (0,0) -- cycle;
      \draw[thick,fill=lightgray] (180:1) -- (300:1) -- (0,0) -- cycle;
      \draw[thick,fill=lightgray] (300:1) -- (60:1) -- (0,0) -- cycle;
      \draw[thick,fill=lightgray] (300:1) -- (60:1) -- (0:2) -- cycle;

      \node[vertex] (z) at (0,0) [label=right:$z$] {};
      \node[vertex] (v) at (60:1) [label=above:$v$] {};
      \node[vertex] (x) at (180:1) [label=left:$x$] {};
      \node[vertex] (w) at (300:1) [label=below:$w$] {};
      \node[vertex] (y) at (0:2) [label=below:$y$] {};
    \end{tikzpicture}
    \caption{The simplicial complex $\Delta$.}
    \label{fig:6}
  \end{figure}
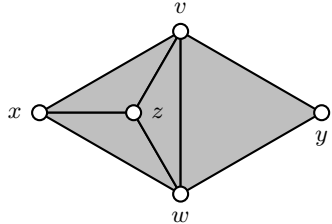
  The $f$-vector of $\Delta$ is $(1,5,8,4)$.  In
  particular, $\Bbbk [\Delta]$ has dimension 3 and multiplicity 4.

  By \Cref{thm:8}, $\mathcal{P}_1 (I)$ is the Stanley--Reisner ideal of
  a simplicial complex $\Gamma_1$ with $f$-vector:
  \begin{equation*}
    \begin{pmatrix}
      1&5&8&4
    \end{pmatrix}
    \begin{pmatrix}
      1&0&0&0&0&0&0\\
      0&2&1&0&0&0&0\\
      0&0&4&4&1&0&0\\
      0&0&0&8&12&6&1
    \end{pmatrix}
    =
    \begin{pmatrix}
      1&10&37&64&56&24&4
    \end{pmatrix}.
  \end{equation*}
  We deduce that $\Bbbk [\Gamma_1]$ has dimension 6 and multiplicity
  1. Similarly, $\mathcal{P}_2 (I)$ is the Stanley--Reisner ideal of a
  simplicial complex $\Gamma_2$ with $f$-vector:
  \begin{equation*}
    \begin{split}
      &\begin{pmatrix}
        1&5&8&4
      \end{pmatrix}
      \begin{pmatrix}
        1&0&0&0&0&0&0&0&0&0\\
        0&3&3&1&0&0&0&0&0&0\\
        0&0&9&18&15&6&1&0&0&0\\
        0&0&0&27&81&108&81&36&9&1
      \end{pmatrix}\\
      =&
      \begin{pmatrix}
        1&15&87&257&444&480&332&144&36&4
      \end{pmatrix}.
    \end{split}
  \end{equation*}
  We deduce that $\Bbbk [\Gamma_2]$ dimension 9 and multiplicity 4.
\end{example}

Recall that, for a homogeneous ideal $I$ in the standard graded
polynomial ring $\Bbbk [X]$, the \emph{Betti numbers} of $\Bbbk [X]/I$
are defined as
\begin{equation*}
  \beta_{i,j} (\Bbbk [X]/I) = \dim_\Bbbk \operatorname{Tor}_i (\Bbbk [X]/I,\Bbbk)_j,
\end{equation*}
i.e., $\beta_{i,j} (\Bbbk [X]/I)$ is the dimension of the degree $j$ component of $\operatorname{Tor}_i (\Bbbk [X]/I,\Bbbk)$ as a vector
space over $\Bbbk$. One often assembles the Betti numbers into the \emph{Betti diagram}, where the $i$-th column for any $i\geq 0$ consists of the Betti numbers $(\beta_{i,i+j}(\Bbbk [X]/I))_{j\geq 0}$.

For our results on Betti numbers of principal $s$-jets, it is also convenient to assemble the Betti numbers into an upper-triangular matrix, where the $i$-th row of the corresponding Betti diagram is slid $i$ units to the right:
 \begin{equation*}
    [\beta_{j-i,j}(\Bbbk[X]/I)]_{0\leqslant i\leqslant j}=\begin{pmatrix}
      \beta_{0,0} & \beta_{1,1} & \beta_{2,2} & \beta_{3,3} & \dots \\
                  & \beta_{0,1} & \beta_{1,2} & \beta_{2,3} & \dots \\
                  & & \beta_{0,2} & \beta_{1,3} & \dots \\
      & & & \ddots & \ddots
    \end{pmatrix}.
  \end{equation*}
The first column of the Betti diagram becomes the main diagonal of the matrix, and the subsequent columns of the Betti diagram become the subsequent diagonals of the matrix.

An invariant that can be easily obtained from the
Betti numbers of $\Bbbk [X]/I$ is the \emph{(Castelnuovo--Mumford)
  regularity} of $\Bbbk [X]/I$, which is defined as
\begin{equation*}
  \max \{j-i : \beta_{i,j} (\Bbbk [X]/I) \neq 0\}.
\end{equation*}
When $I$ is the Stanley--Reisner ideal of a simplicial complex
$\Delta$, the Betti numbers of $\Bbbk [\Delta]=\Bbbk [X]/I$ can be
computed in terms of the homology groups of certain subcomplexes of
$\Delta$ via \emph{Hochster's Formula} (see \cite[\S 5]{MR1251956} or
\cite[\S 1, 5]{MR2110098}).

\begin{theorem}\label{thm:9}
  Let $I \subseteq \Bbbk [X]$ be the Stanley--Reisner ideal of a
  simplicial complex $\Delta$ with vertex set $X$. For
  $s\in \mathbb{N}$, the ideal $\mathcal{P}_s (I)$ of principal
  $s$-jets of $I$ is the Stanley--Reisner ideal of a simplicial complex
  $\Gamma_s$ with vertex set $\mathcal{J}_s (X)$.  The Betti numbers
  of $\Bbbk [\Gamma_s]$ and $\Bbbk [\Delta]$ are related by
  \begin{equation*}
    [\beta_{k-i,k} (\Bbbk [\Gamma_s])]_{0\leqslant i\leqslant k} =
    [\beta_{j-i,j} (\Bbbk [\Delta])]_{0\leqslant i\leqslant j} L_s,
  \end{equation*}
  where $L_s$ is the order $s$ lifting matrix.  In particular,
  $\Bbbk [\mathcal{J}_s(X)]/\mathcal{P}_s(I)$ has the same regularity as $\Bbbk [X]/I$.
\end{theorem}


\begin{proof}
  Given a subset $V\subseteq \mathcal{J}_s (X)$, let $\Gamma_{s,V}$ denote
  the simplicial complex obtained by restricting $\Gamma_s$ to $V$.
  If $x\in \delta_s (V)$, then the restriction of $\Gamma_{s,V}$ to
  $\delta_s^{-1} (\{x\})$ is a simplex because
  $x_0\cdots x_s \not\in \mathcal{P}_s (I)$, and therefore it is
  contractible. Repeating the same argument with all elements of
  $\delta_s (V)$, we deduce that $\Gamma_{s,V}$ is homotopy equivalent
  to $\Delta_{\delta_s (V)}$, the restriction of $\Delta$ to
  $\delta_s (V)$. It follows that $\Gamma_{s,V}$ and
  $\Delta_{\delta_s (V)}$ have the same reduced homology.

  We apply Hochster's formula \cite[Corollary 5.12]{MR2110098}, as well as the fact that we are working with $\Bbbk$-coefficients so cohomology and homology spaces have the same dimension. Then for $0\leqslant i\leqslant k$, we have
  \begin{equation*}
    \beta_{k-i,k} (\Bbbk [\Gamma_s]) =
    \sum_{\substack{V\subseteq \mathcal{J}_s (X),\\|V|=k}}
    \dim_\Bbbk \widetilde{H}_{i-1} (\Gamma_{s,V}; \Bbbk)
    = \sum_{\substack{V\subseteq \mathcal{J}_s (X),\\|V|=k}}
    \dim_\Bbbk \widetilde{H}_{i-1} (\Delta_{\delta_s (V)}; \Bbbk).
  \end{equation*}
  The simplicial complex $\Delta_{\delta_s(V)}$ only depends on the image $U=\delta_s(V)$ of the depolarization map. It then suffices to count the number of $V$ with the same image, and this is given exactly by the lifting function $\ell_s(j,k)$.
  Thus, we have
  \begin{equation*}
    \begin{split}
      \beta_{k-i,k} (\Bbbk [\Gamma_s])
      &= \sum_{j\geqslant 0} \sum_{\substack{U\subseteq X,\\|U|=j}}
      \dim_\Bbbk \widetilde{H}_{i-1} (\Delta_U; \Bbbk) \, \ell_s (j,k)\\
      &= \sum_{j\geqslant 0} \ell_s (j,k) \sum_{\substack{U\subseteq X,\\|U|=j}}
      \dim_\Bbbk \widetilde{H}_{j-(j-i)-1} (\Delta_U; \Bbbk)\\
      &= \sum_{j\geqslant 0} \beta_{j-i,j} (\Bbbk [\Delta]) \, \ell_s (j,k),
    \end{split}
  \end{equation*}
  where the sum is finite because $\ell_s (j,k) = 0$ when
  $j>k$. Therefore, the entry in row $i$ and column $k$ of the matrix
  $[\beta_{k-i,k} (\Bbbk [\Gamma_s])]$ is the product
  of the $i$-th row of
  $[\beta_{j-i,j} (\Bbbk [\Delta])]$ with the $k$-th
  column of the lifting matrix $L_s$. This proves the equality of
  matrices in the statement of the theorem.

  If $\Bbbk [\Delta]$ has regularity $r$, then all
  entries of $[\beta_{j-i,j} (\Bbbk [\Delta])]$ with
  row index $i>r$ are zero, and there is an integer
  $\jmath' \geqslant 0$ corresponding to a nonzero entry
  $\beta_{\jmath'-r,\jmath'} (\Bbbk [\Delta])$ in row $r$
  and column $\jmath'$.  In light of the matrix equality proved
  above, all entries of
  $[\beta_{k-i,k} (\Bbbk [\Gamma_s])]$ with row index
  $i>r$ are also zero. Moreover, we have
  \begin{equation*}
    \beta_{\jmath'-r,\jmath'} (\Bbbk [\Gamma_s]) \geqslant
    \beta_{\jmath'-r,\jmath'} (\Bbbk [\Delta])
    \, \ell_s (\jmath',\jmath')=
    \beta_{\jmath'-r,\jmath'} (\Bbbk [\Delta]) \neq 0,
  \end{equation*}
  which shows $[\beta_{k-i,k} (\Bbbk [\Gamma_s])]_{i,k\geqslant 0}$
  has a nonzero entry in row $r$ and column $\jmath'$.
  Therefore, the regularity of $\Bbbk [\Gamma_s]$
  is $r$.
\end{proof}

\begin{example}\label{exa:6}
  Consider the ideals and simplicial complexes of \Cref{exa:5}. The
  ring $\Bbbk [\Delta]$ has the following Betti diagram.
  \begin{equation*}
    \begin{matrix}
      & 0 & 1 & 2\\
      \text{total:} & 1 & 3 & 2\\
      0: & 1 & . & .\\
      1: & . & 2 & 1\\
      2: & . & 1 & 1
    \end{matrix}
  \end{equation*}
  Multiplying the corresponding matrix of Betti numbers with the
  lifting matrix of the appropriate size gives
  \begin{equation*}
    \begin{split}
      &\begin{pmatrix}
         1 & 0 & 0 & 0 & 0\\
         0 & 0 & 2 & 1 & 0\\
         0 & 0 &0 & 1 & 1
       \end{pmatrix}
        \begin{pmatrix}
          1&0&0&0&0&0&0&0&0\\
          0&2&1&0&0&0&0&0&0\\
          0&0&4&4&1&0&0&0&0\\
          0&0&0&8&12&6&1&0&0\\
          0&0&0&0&16&32&24&8&1
        \end{pmatrix}\\
      =&
         \begin{pmatrix}
           1&0&0&0&0&0&0&0&0\\
           0&0&8&16&14&6&1&0&0\\
           0&0&0&8&28&38&25&8&1
         \end{pmatrix}.
    \end{split}
  \end{equation*}
  Hence, the ring $\Bbbk [\Gamma_1]$ has the following Betti diagram.
  \begin{equation*}
    \begin{matrix}
      & 0 & 1 & 2 & 3 & 4 & 5 & 6\\
      \text{total:} & 1 & 16 & 44 & 52 & 31 & 9 & 1\\
      0: & 1 & . & . & . & . & . & .\\
      1: & . & 8 & 16 & 14 & 6 & 1 & .\\
      2: & . & 8 & 28 & 38 & 25 & 8 & 1
    \end{matrix}
  \end{equation*}
  Similarly, multiplying the matrix of Betti numbers with the second
  order lifting matrix of the appropriate size gives
  \begin{equation*}
    \begin{split}
      &\begin{pmatrix}
         1 & 0 & 0 & 0 & 0\\
         0 & 0 & 2 & 1 & 0\\
         0 & 0 &0 & 1 & 1
       \end{pmatrix}
        \begin{pmatrix}
          1&0&0&0&0&0&0&0&0&0&0&0&0\\
          0&3&3&1&0&0&0&0&0&0&0&0&0\\
          0&0&9&18&15&6&1&0&0&0&0&0&0\\
          0&0&0&27&81&108&81&36&9&1&0&0&0\\
          0&0&0&0&81&324&594&648&459&216&66&12&1
        \end{pmatrix}\\
      =&
         \begin{pmatrix}
           1&0&0&0&0&0&0&0&0&0&0&0&0\\
           0&0&18&63&111&120&83&36&9&1&0&0&0\\
           0&0&0&27&162&432&675&684&468&217&66&12&1
         \end{pmatrix}.
    \end{split}
  \end{equation*}
  Therefore, the ring $\Bbbk [\Gamma_2]$ has the following Betti
  diagram.
  \begin{equation*}
    \begin{matrix}
      & 0 & 1 & 2 & 3 & 4 & 5 & 6 & 7 & 8 & 9 & 10\\
      \text{total:} & 1 & 45 & 225 & 543 & 795 & 767 & 504 & 226 & 67 & 12 & 1\\
      0: & 1 & . & . & . & . & . & . & . & . & . & .\\
      1: & . & 18 & 63 & 111 & 120 & 83 & 36 & 9 & 1 & . & .\\
      2: & . & 27 & 162 & 432 & 675 & 684 & 468 & 217 & 66 & 12 & 1
    \end{matrix}
  \end{equation*}
\end{example}

Let $I$ be a homogeneous ideal in a standard graded polynomial ring
$\Bbbk [X]$. The quotient $\Bbbk [X] / I$ is said to have a
$d$-\emph{linear resolution} for some positive integer $d$ if its only
nonzero Betti numbers $\beta_{i,j} (\Bbbk [X] / I)$ are those with
indices $i=j=0$, or $i>0$ and $j=i+d-1$ (see \cite[Exercise
4.1.17]{MR1251956}). Equivalently, $\Bbbk [X] / I$ has a
$d$-\emph{linear resolution} if and only if $I$ has a minimal set of
generators all of degree $d$ and $\Bbbk [X] / I$ has regularity $d-1$.

\begin{corollary}\label{cor:3}
  Let $I \subseteq \Bbbk [X]$ be a squarefree monomial ideal. Then $\Bbbk[X]/I$ has a $d$-linear resolution if and only if $\Bbbk[\mathcal{J}_s(X)]/\mathcal{P}_s(I)$ has a $d$-linear resolution for every $s\in\mathbb{N}$.
\end{corollary}
  
  

\begin{proof}
Since $\mathcal{P}_0(I)=I$, the backwards implication is immediate.
  In the forward direction,
  observe that for all $s\in\mathbb{N}$, the ideal $\mathcal{P}_s (I)$ is
  minimally generated in degree $d$ by \Cref{cor:2}, and
  $\Bbbk [\mathcal{J}_s(X)]/\mathcal{P}_s(I)$ has regularity $d-1$ by
  \Cref{thm:9}.
\end{proof}

We conclude this section with a combinatorial application to principal
jets of graphs.  Recall that a graph $G$ is \emph{chordal} if it
contains no induced cycle of length four or higher. We say a graph $G$
is \emph{cochordal} if its complement is chordal. A result of Fröberg
\cite{MR1171260} characterizes cochordal graphs as those whose edge
ideals have a linear resolution. We invite the reader to consult
\cite{MR3184120} for an introduction to this material.  \Cref{cor:3}
immediately implies the following result.

\begin{corollary}
  A graph $G$ is cochordal if and only if, for every $s\in\mathbb{N}$,
  the graph $\mathcal{P}_s (G)$ of principal $s$-jets of $G$ is
  cochordal.
\end{corollary}

\section{Open Questions}
\label{sec:open-questions}

We conclude with some questions related to the work of the previous
sections.  The first two are concerned with the primary components of
jets of squarefree monomial ideals and the corresponding associated
primes.

\begin{question}
  Let $C$ be a clutter and let $s\in\mathbb{N}$.  Does the ideal
  $\mathcal{J}_s (I(C))$ have any embedded components?
\end{question}

\begin{question}
  Let $C$ be a clutter and let $s\in\mathbb{N}$. Suppose that we have
  an explicit description for a minimal vertex cover $W$ of
  $\mathcal{J}_s (C)$. What is the primary component of
  $\mathcal{J}_s (I(C))$ corresponding to the minimal prime
  $\langle W\rangle$? What is the multiplicity of the scheme
  $\mathcal{V} (\mathcal{J}_s (I(C)))$ along the irreducible component
  $\mathcal{V} (\langle W\rangle)$?
\end{question}

For the determinantal variety of a generic matrix, there is a
correspondence between certain components of its jet scheme and those
of the jet scheme of a smaller matrix \cite[Theorem
2.8]{MR2100311}. We wonder if a recursive construction might exist for
the components of jets of squarefree monomial ideals, which can be
rephrased combinatorially.

\begin{question}
  Let $C$ be a clutter and let $s\in \mathbb{N}$. Can the minimal
  vertex covers of $\mathcal{J}_s (C)$ be described in terms of the
  minimal vertex covers of some subclutter of $C$?
\end{question}

Consider a \emph{$d$-uniform} clutter $C$ with vertex set $X$, i.e., a
clutter whose edges all contain $d$ vertices. In analogy with the
notion of very well-covered graph, one could say $C$ is very
well-covered if all its minimal vertex covers have cardinality
$\frac{1}{d} |X|$. We could not find this definition in the
literature, so its validity should be investigated before
proceeding. Assuming this definition is acceptable, one might wonder
if \Cref{thm:6} extends to clutters.

\begin{question}
  Are the jets of a very well-covered clutter also very well-covered?
\end{question}

Let $I$ be the Stanley--Reisner ideal of a simplicial complex
$\Delta$. One can define the jets and principal jets of $\Delta$ via the simplicial complexes associated with the squarefree
monomial ideals $\sqrt{\mathcal{J}_s (I)}$ and $\mathcal{P}_s (I)$.

\begin{question}
  What properties of the original complexes do the (principal) jets
  preserve? Can the invariants of (principal) jets of complexes be
  computed in terms of the corresponding invariants of the original
  complexes?
\end{question}

C.~Yuen developed the theory of truncated wedge schemes
\cite{math/0608633}, a higher dimensional analog of jet schemes.  As
is the case for jets, the reduced scheme structure of a truncated
wedge scheme of any monomial scheme is itself a monomial scheme. Thus,
one can define the clutter of $s$-wedges of any clutter in analogy
with \Cref{def:2}.
  
\begin{question}
  What do these clutters of wedges look like? What properties of the
  original clutters do they preserve? Can their invariants be computed
  in terms of the corresponding invariants of the original clutters?
\end{question}


\newcommand{\etalchar}[1]{$^{#1}$}
\def\cprime{$'$} \def\Dbar{\leavevmode\lower.6ex\hbox to 0pt{\hskip-.23ex
  \accent"16\hss}D} \def\Dbar{\leavevmode\lower.6ex\hbox to 0pt{\hskip-.23ex
  \accent"16\hss}D} \def\Dbar{\leavevmode\lower.6ex\hbox to 0pt{\hskip-.23ex
  \accent"16\hss}D} \def\Dbar{\leavevmode\lower.6ex\hbox to 0pt{\hskip-.23ex
  \accent"16\hss}D}

\end{document}